\documentclass{amsart}

\usepackage{booktabs}
\usepackage{array}
\usepackage{url}
\usepackage[colorlinks=true,linkcolor=blue,citecolor=blue,urlcolor=blue]{hyperref}
\usepackage{pgfplots}
\usetikzlibrary{arrows.meta,decorations.pathreplacing,shapes.geometric}
\pgfplotsset{compat=1.16}

\newtheorem{theorem}{Theorem}[section]
\newtheorem{lemma}[theorem]{Lemma}
\newtheorem{proposition}[theorem]{Proposition}
\newtheorem{corollary}[theorem]{Corollary}

\theoremstyle{definition}
\newtheorem{definition}[theorem]{Definition}

\theoremstyle{remark}
\newtheorem{remark}[theorem]{Remark}

\numberwithin{equation}{section}

\DeclareMathOperator{\tr}{tr}
\DeclareMathOperator{\sinc}{sinc}
\DeclareMathOperator{\spn}{spread}

\DeclareUnicodeCharacter{2200}{\ensuremath{\forall}}
\DeclareUnicodeCharacter{2203}{\ensuremath{\exists}}
\DeclareUnicodeCharacter{03B5}{\ensuremath{\varepsilon}}
\DeclareUnicodeCharacter{2080}{\ensuremath{{}_0}}
\DeclareUnicodeCharacter{211D}{\ensuremath{\mathbb R}}
\DeclareUnicodeCharacter{2265}{\ensuremath{\ge}}
\DeclareUnicodeCharacter{2264}{\ensuremath{\le}}
\DeclareUnicodeCharacter{2192}{\ensuremath{\to}}

\begin{document}

\title[Distinct and simple zeros of the Riemann zeta function]{More than 83.9\% of the zeros of the Riemann zeta function are distinct and more than 67.35\% are simple and on the critical line}

\author[K. M. Knausg{\aa}rd]{Kristian Muri Knausg{\aa}rd}
\address{Kristiansand, Norway}
\email{kristianmk@ieee.org}



\begin{abstract}
Let $N(T)$ count the nontrivial zeros of the Riemann zeta function up to
height $T$ with multiplicity, $N_d(T)$ the distinct ones, and $N_0^s(T)$ those
that are simple and on the critical line. We prove $\liminf
N_d(T)/N(T)\ge1645064/1960733=0.83900\ldots$; earlier work proves
$0.83699\ldots$, and a report we have not verified claims $0.83805\ldots$.
Hence more than $67.80\%$ are simple. Also
$\liminf N_0^s(T)/N(T)\ge0.67353\ldots$; earlier work proves
$0.67250\ldots$, and reports we have not verified claim $0.673492$. Two estimates of Lamzouri are improved: at least $88.93\%$ of the
zeros are simple or on the critical line, and the proportions of simple zeros
and of zeros on the line average at least $83.67\%$. By the unconditional form
of Montgomery's pair-correlation theorem the energy, a sum of a test function
over pairs of zeros, is known asymptotically. A zero of multiplicity $d$
contributes $d^2$ to it and nearby zeros contribute too, so a lower bound for
such pairs leaves less energy for multiple zeros. The proof has three
steps. First, for zeros at least a fixed fraction of the mean
spacing apart, a large sieve inequality lets their pairs be used in full, with multiplicities. Second, their
contribution is bounded below by a computer-assisted inequality for seven or
eight consecutive zeros which distinguishes simple and double zeros. Its correction terms telescope,
as increments of a storage function. Third, the test function is chosen to
make this contribution large, at the cost of more energy. We also prove upper
limits for what such inequalities can give with the two main test functions.
The lower bounds and these limits are proved in Lean~4, analytic inputs
included. The proofs use three axioms beyond those of Mathlib, one per
inequality. Each records that a search program returned true;
Lean's kernel does not check the run itself. This work is an experiment in
AI-assisted mathematical research.
\end{abstract}

\maketitle

\section{Introduction}

Let $N(T)$ count the nontrivial zeros $\rho=\sigma+i\gamma$ of the Riemann zeta
function with $0<\gamma\le T$, with multiplicity, and let $N_d(T)$ count the
distinct ones. The simple zeros conjecture states that every nontrivial zero
of the zeta function is simple, so that $N_d(T)=N(T)$ for all $T$. It is not
known to follow from the Riemann hypothesis, and it is not known to imply it.
Under the Riemann hypothesis, Montgomery's pair correlation conjecture implies
that almost all zeros are simple~\cite{Mon73}. For $0<\gamma\le3\cdot10^{12}$
the Riemann hypothesis has been verified by counting sign changes on the
critical line~\cite{PT21}; the number of sign changes found equals $N(T)$,
which counts with multiplicity, and it follows that these zeros are simple.

Under the Riemann hypothesis more is known: at least $\frac{19}{27}$ of the
zeros are simple~\cite{BHB13} and at least $0.8477$ of them are
distinct~\cite{CGdL}. The bounds of this paper are unconditional: they do
not assume the Riemann hypothesis or any other unproved statement.

Unconditionally, Farmer~\cite{Far95} proved that more than $63.9\%$ of the
zeros are distinct and Ki and Lee~\cite[Theorem~5]{KiLee12} that more than
$70\%$ are; see also Wu~\cite{Wu12}. An argument discovered by
Claude (Anthropic)~\cite{Claude26} and verified by Alp\"oge and
Furman~\cite{AF}, together with a second proof by Lamzouri~\cite{Lam}, gave
$\liminf N_d(T)/N(T)\ge0.83625\ldots$, and Wang~\cite{Wang} improved the
constant by about $3\cdot10^{-8}$. Chandru~\cite{Chandru} obtained $0.83652$
with a computer-assisted inequality for seven points, in a preprint that has
not been refereed; the reports~\cite{Schaefer,Singh} state $0.83671$ and
$0.83672$. The first paper in this series~\cite{K1} proved
$0.8369928\ldots$. Santiba\~nez-Leal~\cite{SantLeal} applied its argument to
stronger local inequalities and states $0.83737\ldots$. A recent report of
D'Souza~\cite{DSouza} states $0.83757\ldots$, with a formal proof in Lean, and
$0.83805\ldots$, where one finite inequality is checked outside Lean. We have
not verified these reports.

\begin{theorem}\label{thm:main}
$\displaystyle\liminf_{T\to\infty}\frac{N_d(T)}{N(T)}\ \ge\
q:=\frac{1645064}{1960733}=0.83900459675\ldots$
\end{theorem}

The gain over~\cite{K1} is $0.0020$, and the larger value of~\cite{DSouza}
is exceeded by $0.0009$.

\begin{corollary}\label{cor:simple}
Let $N_s(T)$ count the simple nontrivial zeros with $0<\gamma\le T$. Then
\[
 \liminf_{T\to\infty}\frac{N_s(T)}{N(T)}\ \ge\ 2q-1
 =\frac{1329395}{1960733}=0.67800919\ldots
\]
\end{corollary}

\begin{proof}
A zero that is not simple has multiplicity at least two, so
$N\ge N_s+2(N_d-N_s)$, that is, $N_s\ge2N_d-N$.
\end{proof}

Corollary~\ref{cor:simple} counts simple zeros anywhere in the critical strip.
For this quantity~\cite{K1} gives $0.67398$ in the same way,
and~\cite{DSouza} states $0.67515\ldots$ with a formal proof and
$0.67610\ldots$ with the finite inequality just mentioned. The zeros that are simple and on the
critical line form a smaller set. Before 2026 the unconditional bounds for
zeros on the critical line came from the method of
Levinson~\cite{Lev74,Con89}. The last of them is slightly more than
$\frac5{12}$, and $0.4075$ for zeros that are simple and on the
line~\cite{PRZZ}. For the latter, \cite{AF,Lam} proved $0.67250\ldots$,
and~\cite{Wang} added about $7\cdot10^{-8}$. Larger values, obtained along
the same lines, are stated in work that has not been refereed: $0.67300$~\cite{Ainta},
$0.67305$~\cite{Chandru}, $0.67319$~\cite{Tawan}, $0.67331$~\cite{Trmdy,Shi},
$0.67337$~\cite{DSouza}, $0.673399$~\cite{Dev}, $0.6734$~\cite{ZTian},
$0.67341$~\cite{AMTOPA}, $0.673425$~\cite{Schaefer}, $0.673451$~\cite{Singh},
$0.673476$ to $0.673477$~\cite{GebRe,GebCap,Geb,Rakitin} and
$0.673492$~\cite{Freund}, the last in a draft pull request to a public git
repository; the survey in~\cite[Section~1.4]{DSouza} names
further ones. The values of~\cite{ZTian} are stated there as complete formal
proofs in Lean. A public list of formal proofs, which admits the three
standard axioms of Lean only, has entries up to $0.67348$~\cite{Dani,Lavery};
they are stated for the number of distinct zeros on the critical line, which
is at least $N_0^s(T)$, and the proof in~\cite{Lavery} obtains the bound for
$N_0^s(T)$ first. We have not verified these reports, and no result of this
paper depends on them: they are cited for comparison and for the origin of
ideas. All of the values are
smaller than the bound of Theorem~\ref{thm:line} below, which holds for that
number as well. A report~\cite{YY} claims much larger values,
$0.7962$ for this quantity and $0.8981$ for distinct zeros, from trace moments
of order up to six. Its authors grade it as a candidate whose analytic part is
not yet verified. In~\cite[Section~7.2]{AF} bounds of this kind, such as
$\frac{13}{18}$ from the fourth moment, are derived from a hypothesis on these
moments. We do not count this report among the values above. Neither Theorem~\ref{thm:main} nor Corollary~\ref{cor:simple} says
anything about the critical line: a zero off the line is distinct, and it may
be simple. For the critical line the method is run with the clipping level
$c=2$ in place of $c=\frac{67}{20}$ and with a second window, and then it
gives the following.

\begin{theorem}\label{thm:line}
Let $N_0^s(T)$ count the zeros $\rho=\frac12+i\gamma$ with $0<\gamma\le T$ that
are simple. Then
\[
 \liminf_{T\to\infty}\frac{N_0^s(T)}{N(T)}\ \ge\ s_0:=\frac{1669159}{2478195}=0.67353820\ldots
\]
\end{theorem}

The gain over the bound $0.67250$ of~\cite{AF,Lam} is $0.00103$, and the gain over the
largest of the values above is $4.5\cdot10^{-5}$.

Lamzouri~\cite{Lam} bounded two further quantities by the same method. Let
$N_0(T)$ count the zeros on the critical line with $0<\gamma\le T$, with
multiplicity, and let $N_{s\cup0}(T)$ count, with multiplicity, the zeros with
$0<\gamma\le T$ that are simple or on the critical line.
By~\cite[Theorem~1.1]{Lam}, the average of the proportion of simple zeros and
the proportion of zeros on the line is at least $0.83625\ldots$, and the
proportion of zeros that are simple or on the line is at least
$0.88762\ldots$. For the second quantity Chandru~\cite[(1.10)]{Chandru} states
$0.88855$, from a seven-point inequality for simple points, in the preprint
mentioned above, and~\cite{DSouza} states $0.88791$. The certificates behind Theorems~\ref{thm:main}
and~\ref{thm:line} give more for both quantities.

\begin{theorem}\label{thm:further}
With $s_0$ as in Theorem~\ref{thm:line},
\begin{align*}
 \liminf_{T\to\infty}\frac{N_s(T)+N_0(T)}{2N(T)}&\ \ge\ \frac{1+s_0}2
 =\frac{2073677}{2478195}=0.83676910\ldots,\\
 \liminf_{T\to\infty}\frac{N_{s\cup0}(T)}{N(T)}&\ \ge\ \frac{3376433}{3797325}=0.88916092\ldots
\end{align*}
\end{theorem}

The first bound comes from the second window and the second bound from the
first window, and no further computation is needed
(Section~\ref{sec:further}). The first window was not made for the second
quantity. A \emph{third window}, with a seven-point inequality found for this
quantity, gives slightly more.

\begin{theorem}\label{thm:third}
\[
 \liminf_{T\to\infty}\frac{N_{s\cup0}(T)}{N(T)}\ \ge\ \frac{50652159}{56956375}=0.88931500\ldots
\]
\end{theorem}

In particular the larger of $N_s(T)$ and $N_0(T)$
is at least $(0.8367+o(1))N(T)$, and the zeros that are off the line and not
simple have proportion at most $0.1107$. The proofs of the
four theorems and of the corollary are formalised (Section~\ref{sec:lean}).

\subsection*{The idea}
Rescale the zeros so that their mean spacing is one. Fix an even probability
density $f$ on $I=[-\frac12,\frac12]$, the \emph{window}, with Fourier
transform $K(t)=\int f(u)e^{-2\pi itu}\,du$. A zero on the critical line at
rescaled height $y$ is represented by the wave packet
$\sqrt{f(u)}\,e^{2\pi iuy}$, a unit vector in $L^2(I)$, and two such vectors
have inner
product $K(y-y')$. Summing the rank-one projections over the zeros on the
line, with multiplicity, and adding a Hermitian term for each pair of zeros
off the line, gives an operator $A$ with $\tr A=N$~\cite{AF,Lam}. Two numbers are known
about $A$. The trace counts the zeros. The \emph{energy} $\tr A^2$ is given by
Montgomery's pair-correlation theorem in its unconditional
form~\cite{Mon73,BGST}: after a smoothing of $f$ that is removed at the end,
\begin{equation}\label{eq:energy-intro}
 \tr A^2=(C(f)+o(1))\,N,\qquad C(f)=\int f^2+\iint|s-t|\,f(s)f(t)\,ds\,dt .
\end{equation}
The energy is a sum over pairs of zeros. If the zeros on the line are at
$y_1<\dots<y_l$ with multiplicities $d_i\in\{1,2\}$, their part of it is
\begin{equation}\label{eq:pair-intro}
 \underbrace{\sum_i d_i^2}_{\text{repeated zeros}}\ +\
 \underbrace{\sum_{i\ne j}d_id_j\,K(y_i-y_j)^2}_{\text{different zeros}} .
\end{equation}
A zero of multiplicity $d$ contributes $d$ to the trace and $d^2$ to the first
sum, so repeated zeros cost energy, and the energy is fixed
by~\eqref{eq:energy-intro}. Keeping only the first sum gives
$\tr A^2\ge3N-2N_d$ and hence $\liminf N_d/N\ge(3-C(f))/2$. With the window
that minimises $C(f)$, due to Montgomery and Taylor~\cite{Mon75}, this is the
bound $0.83625$ of~\cite{AF,Lam}.

The count behind this is $d^2\ge3d-2$, with equality for $d=1$ and $d=2$: a
zero of multiplicity $d$ counts $d$ in $N$ and $1$ in $N_d$. Zeros off the
line and zeros of multiplicity three or more are handled by a lemma of linear
algebra (Lemma~\ref{lem:threshold}), which shows that they cannot bring the
energy below this count. The lemma has a price. The zeros that are left enter
through $\tr\varphi_c(G)$ and not through $\tr G^2$, where $G$ is their Gram
matrix and $\varphi_c$ replaces the square of an eigenvalue above a level $c$
by its tangent line. We call this \emph{clipping}. The second sum
in~\eqref{eq:pair-intro} is $\tr G^2-\sum_id_i^2$, so it can be used in full
only where $G$ has no eigenvalue above $c$.

The second sum is the energy of the interaction between different zeros. For
a suitable window it cannot be small: there are $N$ units of multiplicity on a
line of length $N$, so most zeros have neighbours within a few mean spacings,
and a double zero interacts twice as strongly with its neighbours as a simple
one. The flat window shows why the window matters: its kernel vanishes at the
nonzero integers, and the unit lattice has no interaction at all. Every unit
of interaction energy is a unit that is not available for repeated zeros.
\emph{The idea is to prove a lower bound for the interaction energy that is
linear in the number of simple points, the number of double points and the
length of the line, and to choose the window so that this bound is large.}

The paper~\cite{K1} proved such a bound of small size. There the second sum
entered through a matrix inequality, the main result of that
paper~\cite[Theorem~2]{K1}. It bounds the clipped energy from below by a
function of the Gram matrix $(K(y_i-y_j))$ of the points without their
multiplicities, and the term so obtained was called the \emph{defect}. That
inequality is not used here, for two reasons. It counts every pair with the
weight $1$ in place of $d_id_j$, so it cannot reward a double zero more than a
simple one. And its lower bound grows only linearly above a second level
$1+\tau$, which limits the argument to blocks of about $1300$ points, with a
loss in every block. For points that keep a fixed distance from each other
neither loss is necessary: the clipped energy is then exactly the second sum
in~\eqref{eq:pair-intro}, with the weights (item~(1) below). From~\cite{K1} we
keep the threshold lemma, the operator and the count of the zeros
(Lemma~\ref{lem:threshold}, Section~\ref{sec:proof}). Its local input, an
inequality for seven consecutive points without multiplicities imported from
work on simple zeros~\cite{Ainta,Trmdy,Shi,Tawan}, is replaced by one with
multiplicities that is proved here. Table~\ref{tab:papers} puts the two
papers side by side, and Figure~\ref{fig:method} shows the argument in one
picture, for zeros on the line of multiplicity at most two.

\begin{figure}
\begin{tikzpicture}[x=1cm,y=1cm,
  pt/.style={circle,fill=black,inner sep=0pt,minimum size=4pt},
  dbl/.style={circle,fill=black,inner sep=0pt,minimum size=8pt},
  box/.style={draw=black,line width=0.5pt,inner sep=4pt,align=center,text width=8.6cm,minimum height=1.15cm}]
\fill[black!15] (4.25,-0.45) rectangle (10.95,1.75);
\draw[line width=0.5pt] (0,0)--(12.6,0);
\foreach \x in {0.8,1.8,4.6,5.6,6.6,7.6,9.6,10.6,11.6,12.3} \node[pt] at (\x,0) {};
\node[dbl] at (2.9,0) {};
\node[dbl] at (8.6,0) {};
\foreach \a/\b in {4.6/5.6,5.6/6.6,6.6/7.6,7.6/8.6,8.6/9.6,9.6/10.6}
  \draw[line width=1.5pt] (\a,0.08) to[out=62,in=118] (\b,0.08);
\foreach \a/\b in {4.6/6.6,5.6/7.6,6.6/8.6,7.6/9.6,8.6/10.6}
  \draw[line width=1pt] (\a,0.08) to[out=58,in=122] (\b,0.08);
\foreach \a/\b in {4.6/7.6,5.6/8.6,6.6/9.6,7.6/10.6}
  \draw[line width=0.5pt] (\a,0.08) to[out=56,in=124] (\b,0.08);
\draw[line width=0.5pt,dashed] (4.6,0.08) .. controls (6.6,2.1) and (9.6,2.1) .. (11.6,0.08);
\node[anchor=south] at (7.3,1.8) {a frame of seven consecutive zeros};
\node[anchor=north] at (7.6,-0.5) {overlaps $K(y_j-y_i)^2$ of the pairs inside the frame};
\node[anchor=south] at (11.85,1.0) {dropped};
\node[anchor=south] at (1.3,0.2) {simple zeros};
\draw[decorate,decoration={brace,mirror,amplitude=3pt},line width=0.5pt] (0.8,-0.2)--(1.8,-0.2);
\node[anchor=north] at (1.3,-0.35) {mean spacing $1$};
\draw[line width=0.5pt,-{Stealth[length=5pt]}] (2.9,1.05)--(2.9,0.22);
\node[anchor=south,align=center] at (2.7,1.05) {double zero:\\ energy $2^2=4$};
\node[box] (a) at (6.3,-2.2) {Pair correlation gives the total energy:\\ \mbox{$\operatorname{tr}A^2=(C(f)+o(1))\,N$}};
\node[box] (b) at (6.3,-3.75) {Repeated zeros cost energy:\\ \mbox{$\operatorname{tr}A^2\ge n_1+4n_2+\text{interaction}$}};
\node[box] (c) at (6.3,-5.30) {Seven-point inequality, summed over all frames:\\ \mbox{$\text{interaction}\ge x_1n_1+x_2n_2-b\,N$}};
\node[box] (d) at (6.3,-6.85) {So few zeros are double:\\ \mbox{$\liminf N_d/N\ge(3-C+x_2-x_1-b)/(2-2x_1+x_2)$}};
\draw[line width=0.5pt,-{Stealth[length=5pt]}] (a)--(b);
\draw[line width=0.5pt,-{Stealth[length=5pt]}] (b)--(c);
\draw[line width=0.5pt,-{Stealth[length=5pt]}] (c)--(d);
\end{tikzpicture}
\caption{The method. The zeros are points on a line with mean spacing one. A
zero of multiplicity $d$ carries the energy $d^2$, and two zeros at distance
$t$ carry the energy $2d\,d'K(t)^2$. The total is known from pair
correlation~\cite{Mon73,BGST}, up to a factor $1+o(1)$
(Proposition~\ref{prop:energy}). A local inequality bounds from below the
pair energy inside a frame of seven consecutive zeros
(Proposition~\ref{prop:local}); pairs beyond the frame are dropped. The frame
drawn here contains a double zero; simple and double zeros have separate
rewards. The arcs are drawn up to index distance three only. Summed over all positions of the
frame, the inequality gives a lower bound for the interaction that is linear
in the numbers $n_1$, $n_2$ of simple and double zeros
(Lemma~\ref{lem:transfer}, Theorem~\ref{thm:gram}), and then the energy that
is left limits the number of double zeros (Theorem~\ref{thm:general}). The
boxes are for zeros on the line of multiplicity at most two, without error
terms; Section~\ref{sec:proof} has the general case.}
\label{fig:method}
\end{figure}
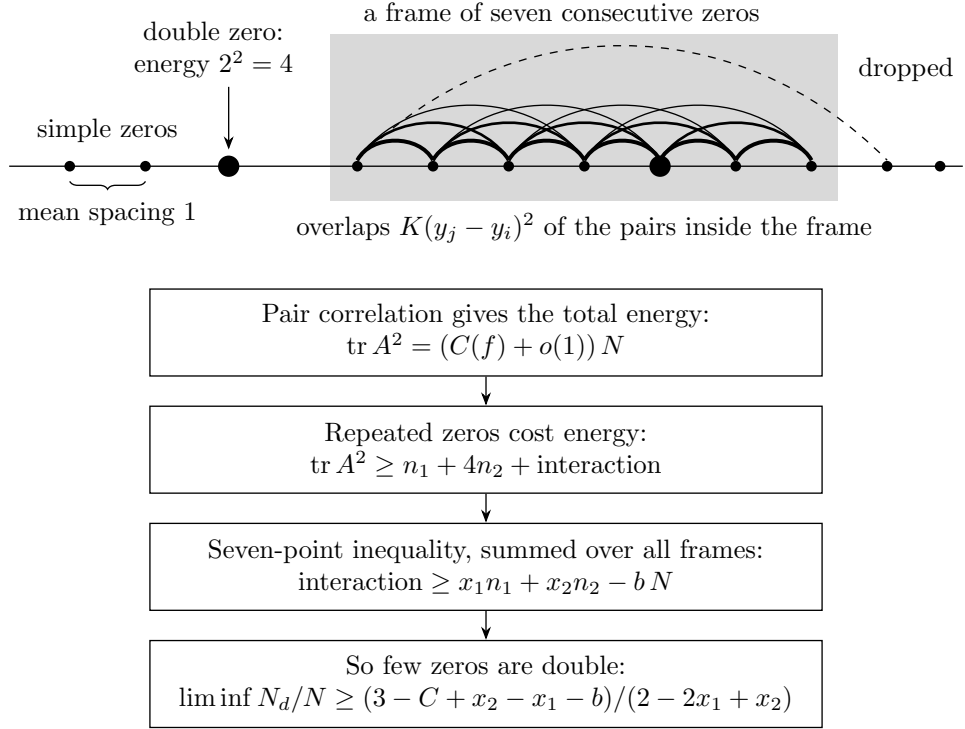

\begin{table}
\caption{The first paper of this series and the present one. Both prove a
linear lower bound for the interaction energy and insert it into the same
count.}
\label{tab:papers}
\begin{tabular}{>{\raggedright\arraybackslash}p{0.22\textwidth}>{\raggedright\arraybackslash}p{0.34\textwidth}>{\raggedright\arraybackslash}p{0.34\textwidth}}\toprule
 & First paper~\cite{K1} & This paper\\\midrule
matrix step & bounds the energy of a block of about $1300$ points from below,
with a loss & exact for separated points: every pair is counted, with the
weight $d_id_j$\\\addlinespace
local inequality & seven points without labels, imported & seven points with
labels and with a storage function, proved here\\\addlinespace
window & the Montgomery--Taylor window with a small correction; $C(f)$ close
to its minimum & a trigonometric polynomial chosen for the interaction; $C(f)$
larger by $0.005$\\\addlinespace
certificate $(x_1,x_2,b)$ & $(0.00443,\,0.00443,\,0.00218)$ &
$(0.05435,\,0.06943,\,0.03731)$\\\addlinespace
gain over $(3-C(f))/2$ & $0.00076$ & $0.00536$\\\addlinespace
bound for $\liminf N_d/N$ & $0.836992$ & $0.839004$\\\addlinespace
formal proof in Lean & the matrix inequalities and the count; the operator,
the energy and the local inequality are hypotheses & the theorem itself, from
the axioms of Mathlib and one computed statement (Section~\ref{sec:lean})\\\bottomrule
\end{tabular}
\end{table}

\subsection*{What is new}
\begin{enumerate}
\item \emph{Separation removes the clipping.} Points closer than
$\theta=\frac{17}{20}$ are paired off and treated as two-by-two blocks. The
remaining points are $\theta$-separated. Remark~\ref{rem:c} says where the
value of $\theta$ comes from. For such points the Gram matrix
$(K(y_i-y_j))$ is bounded by a constant, by Selberg's majorant argument for
the large sieve~\cite{Mon78,Vaa85}; this is the direct half of an inequality
of Ingham type~\cite{Ing36,KL05}. As a consequence the clipped energy of a
separated block is exactly the interaction sum in~\eqref{eq:pair-intro}, with
the weights $d_id_j$ (Lemma~\ref{lem:maj}, Theorem~\ref{thm:gram}). For simple zeros on the critical
line this step is due to G.~Freund~\cite{Freund}: there the zeros closer
than $\frac45$ are paired off, the Gram matrix of the rest is bounded by a
majorant of the same kind, and the seven-point inequality of~\cite{Geb} then
gives the value $0.673492$ quoted above without the loss of the block
argument of~\cite{Ainta}. Disjoint pairs of close zeros were used before
in~\cite{ShiSJ}. D'Souza~\cite{DSouza} uses the step for points with
multiplicities, with a majorant of Beurling--Selberg type. The value
of~\cite{Freund} is stated in a draft pull request to a public git repository
and not in a preprint; we learned of it from the survey
in~\cite[Section~1.4]{DSouza} when the proofs of this paper were complete, and
had found the step independently.
\item \emph{A local inequality with a storage function.} For seven consecutive
separated points with labels $d_i\in\{1,2\}$ we prove, with computer
assistance, an inequality between their spread, their weighted squared overlaps
and a reward which is larger for double points
(Proposition~\ref{prop:local}). The weights depend on the pair of labels. The
inequality also contains functions of the gaps which sum to zero over the six
positions; when the inequality is added over all sets of seven consecutive
points these terms telescope (Lemma~\ref{lem:transfer}). For points without
labels the lemma with such a correction is in~\cite[Section~3.3]{Chandru},
and a seven-point certificate with a nonzero correction, the increment of a
function of five gaps, is in~\cite{Tawan}; certificates with a function of
the state of a frame, and the bound for the loss by its oscillation, are also
in~\cite{Trmdy}. Local inequalities over all patterns of
multiplicities one and two were found independently by
D'Souza~\cite[Definition~5.1]{DSouza}. Our
corrections are sums of functions of one gap, which depend on the labels and
are bounded. For Theorem~\ref{thm:line} the inequality is for eight
consecutive points without labels (Proposition~\ref{prop:local2}); frames of
nine points are used in~\cite{SantLeal,Shi,Trmdy} and frames of eight
in~\cite{Lavery}.
\item \emph{A window for the interaction.} For Theorem~\ref{thm:main} we take
$f(u)=1+\sum_{j=1}^6a_j\cos(2\pi ju)$, which we call the \emph{first window};
the \emph{second window}, for Theorem~\ref{thm:line}, and the \emph{third
window}, for Theorem~\ref{thm:third}, have the same form with other
coefficients, the second with ten terms. For the first window the plain bound $(3-C(f))/2=0.83364\ldots$
is \emph{smaller} than that of the Montgomery--Taylor window, but the
interaction is worth about $0.0054$ instead of about $0.0008$. Windows with
further cosine terms, chosen with the local inequality in view, are used for
simple zeros in~\cite{AMTOPA,Geb,Lavery,Trmdy}.
\end{enumerate}

Two of these devices have names elsewhere. The correction terms of the seven-point
inequality are the increments of a storage function in the sense of
Willems~\cite{Wil72}: the inequality is a dissipation inequality, and summing
it along the line leaves only a boundary term (Remark~\ref{rem:control}). And
the search for the best certificate is dual to an average-cost optimal-control
problem, whose value periodic trial configurations bound from the other side
(Remark~\ref{rem:bellman}). In~\cite{Tawan} the correction is called a
Bellman coboundary, and~\cite{Trmdy} speaks of transfer-operator
certificates.

\subsection*{One formula}
Call $(x_1,x_2,b)$ a \emph{certificate} if the interaction energy, in the
clipped form defined in Section~\ref{sec:proof}, of every configuration with
$n_1$ simple and $n_2$ double points is at least
$x_1n_1+x_2n_2-b\cdot\text{length}$, up to bounded and negligible terms. Then
(Theorem~\ref{thm:general})
\[
 \liminf_{T\to\infty}\frac{N_d(T)}{N(T)}\ \ge\ \frac{3-C(f)+x_2-x_1-b}{2-2x_1+x_2}.
\]
The bound of~\cite{AF,Lam} is the certificate $(0,0,0)$. The bound
of~\cite{Chandru} is the value of this formula at
$(\frac{19}{5000},0,\frac1{500})$, where only the simple points interact. The
bound of~\cite{K1} is a certificate
with $x_1=x_2$, because its argument does not distinguish simple from double
zeros. Here $(x_1,x_2,b)=(0.054350,\,0.069433,\,0.037314)$.

For zeros that are simple and on the critical line only the simple zeros on
the line may interact, and a pair $(x,b)$ with interaction energy at least
$x\,n-b\cdot\text{length}$ for $n$ simple points gives
$\liminf N_0^s(T)/N(T)\ge(2-C(f)-b)/(1-x)$
(Theorem~\ref{thm:line-general}). The bound of~\cite{AF,Lam} is the pair
$(0,0)$, and in Theorem~\ref{thm:line}, $(x,b)=(0.008722,\,0.004389)$.

\begin{figure}
\begin{tikzpicture}[x=1cm,y=1cm,
  proved/.style={circle,fill=black,inner sep=0pt,minimum size=6pt},
  stated/.style={circle,draw=black,fill=white,line width=0.5pt,inner sep=0pt,minimum size=6pt},
  ceil/.style={diamond,draw=black,fill=white,line width=0.5pt,inner sep=0pt,minimum size=7pt}]
\node[proved] at (0.3,3.9) {}; \node[anchor=west] at (0.45,3.9) {proved};
\node[stated] at (2.6,3.9) {}; \node[anchor=west] at (2.75,3.9) {stated, not refereed};
\node[ceil] at (7.4,3.9) {}; \node[anchor=west] at (7.57,3.9) {ceiling};
\node[anchor=west] at (0,3.0) {(a)};
\draw[line width=0.5pt,dashed] (10.247,-2.1)--(10.247,1.9);
\node[anchor=south] at (10.247,1.9) {$84\%$};
\draw[line width=0.5pt] (0.2,0)--(12.4,0);
\foreach \v/\x in {0.834/1.635,0.836/4.506,0.838/7.376}
  {\draw[line width=0.5pt] (\x,0)--(\x,-0.1); \node[anchor=north] at (\x,-0.1) {$\v$};}
\node[proved] (p1) at (0.678,0) {}; \node[proved] (p2) at (4.865,0) {};
\node[stated] (p5) at (5.265,0) {}; \node[proved] (p3) at (5.931,0) {}; \node[proved] (p4) at (8.818,0) {};
\foreach \x in {5.529,5.547,6.486,6.773} \node[stated] at (\x,0) {};
\node[stated] (p6) at (7.450,0) {};
\draw[line width=0.5pt] (p6)--(7.450,0.5); \node[anchor=south west,align=left] at (7.360,0.5) {$0.83805$\\ \cite{DSouza}};
\draw[line width=0.5pt] (p1)--(0.678,0.5); \node[anchor=south west,align=left] at (0.2,0.5) {$5/6$\\ \cite{AF}};
\draw[line width=0.5pt] (p2)--(4.865,0.5); \node[anchor=south east,align=right] at (4.95,0.5) {$0.83625$\\ \cite{AF,Lam}};
\draw[line width=0.5pt] (p5)--(5.265,1.6); \node[anchor=south,align=center] at (5.265,1.6) {$0.83652$ \cite{Chandru}};
\draw[line width=0.5pt] (p3)--(5.931,0.5); \node[anchor=south west,align=left] at (5.85,0.5) {$0.83699$\\ \cite{K1}};
\draw[line width=0.5pt] (p4)--(8.818,1.9); \node[anchor=south,align=center] at (8.818,1.9) {$0.83900$\\ Theorem~\ref{thm:main}};
\node[ceil] (c1) at (9.029,0) {}; \node[ceil] (c2) at (10.218,0) {}; \node[ceil] (c3) at (11.582,0) {};
\draw[line width=0.5pt] (c1)--(9.029,-0.75); \node[anchor=north east,align=right] at (9.35,-0.75) {$0.83916$\\ seven points};
\draw[line width=0.5pt] (c2)--(9.95,-1.85); \node[anchor=north east,align=right] at (10.1,-1.8) {$0.83998$, any frame length};
\draw[line width=0.5pt] (c3)--(11.582,-0.75); \node[anchor=north east,align=right] at (12.6,-0.75) {$0.84093$\\ any window};
\begin{scope}[yshift=-5.6cm]
\node[anchor=west] at (0,2.55) {(b)};
\draw[line width=0.5pt] (0.2,0)--(12.4,0);
\foreach \v/\x in {0.6724/1.635,0.6728/4.506,0.6732/7.376,0.6736/10.247}
  {\draw[line width=0.5pt] (\x,0)--(\x,-0.1); \node[anchor=north] at (\x,-0.1) {$\v$};}
\node[proved] (q1) at (2.360,0) {};
\foreach \x in {5.999,6.357,7.326,8.186,8.216,8.623,8.805,8.812,8.930,8.992,9.180,9.357,9.369} \node[stated] at (\x,0) {};
\node[stated] (g1) at (9.475,0) {};
\draw[line width=0.5pt] (g1)--(9.325,1.45); \node[anchor=south east] at (9.495,1.42) {$0.67349$ \cite{Freund}};
\node[proved] (q3) at (9.804,0) {};
\draw[line width=0.5pt] (q1)--(2.360,0.5); \node[anchor=south,align=center] at (2.360,0.5) {$0.67250$\\ \cite{AF,Lam}};
\draw[decorate,decoration={brace,amplitude=4pt},line width=0.5pt] (5.9,0.3)--(9.439,0.3);
\node[anchor=south,align=center] at (6.3,0.5) {$0.67300$ to $0.67347$,\\ the values listed in the text};
\draw[line width=0.5pt] (q3)--(9.804,1.6); \node[anchor=south west,align=left] at (9.711,1.6) {$0.67353$\\ Theorem~\ref{thm:line}};
\node[ceil] (d1) at (10.004,0.24) {}; \node[ceil] (d2) at (11.838,0) {};
\draw[line width=0.5pt] (d1)--(10.274,0.6); \node[anchor=west,align=left] at (10.254,0.8) {$0.67357$\\ eight points};
\draw[line width=0.5pt] (d2)--(11.838,-0.75); \node[anchor=north east,align=right] at (12.6,-0.75) {$0.67383$\\ any frame\\ length};
\end{scope}
\end{tikzpicture}
\caption{What is proved and where the method stops: (a)~distinct zeros,
$\liminf N_d(T)/N(T)$; (b)~simple zeros on the critical line,
$\liminf N_0^s(T)/N(T)$. Filled circles are proved lower bounds. Open circles
are values stated in work that has not been refereed and that we have not
verified. Diamonds are ceilings, that is, upper bounds for what a certificate
of the kind used here can give: with frames of seven points in panel~(a)
and of eight points in panel~(b), with frames of
any length and, in panel~(a), for any window $f$ with $V(f)\le10\,B(f)$ in the
notation of Section~\ref{sec:allwindows}. They limit the method and are not
bounds for the proportions of zeros. In panel~(b) the open circles are the
values listed before Theorem~\ref{thm:line}, of which the largest is
labelled, and the first ceiling is drawn just above the axis. The
ceilings of panel~(b) are for the scheme of Section~\ref{sec:line} only. The
scale of panel~(b) is about five times finer than that of panel~(a).}
\label{fig:ladder}
\end{figure}

\subsection*{How far the method goes}
For a fixed window the question of the best constant is a ground-state
problem for a one-dimensional gas with two kinds of particles, simple points
and double points, and pair interaction $d_id_jK(y_i-y_j)^2$ at prescribed
density. Trial states bound from above what any argument of this type can
prove. Section~\ref{sec:limits} has the proofs and Table~\ref{tab:ceilings}
there the numbers for distinct zeros; the conclusions are these. The bound of
Theorem~\ref{thm:main} is larger than the ceilings for the windows used
before, so a new window was necessary for it. For the first window less than
$0.0010$ remains, so it cannot give $84\%$, and within our scheme inequalities
for seven consecutive points cannot give more than $0.83916$. For every window
$f$ with $V(f)\le10\,B(f)$, a regularity condition that holds for the
windows of the table and for the second window, the ceiling is $0.84093$
(Corollary~\ref{cor:allwindows}), so less than $0.002$ is left. This last
ceiling rests on a construction of Alp\"oge and Furman~\cite{AF}, with a law
on periodic configurations that we computed and certified ourselves
(Section~\ref{sec:allwindows}).

For zeros that are simple and on the critical line, Theorem~\ref{thm:line}
proves $0.673538$ with the second window, inequalities for eight consecutive
points cannot give more than $0.67357$, and no lower bound for the pair energy
of simple points can give more than $0.67383$
(Proposition~\ref{prop:line-ceiling}). Upper bounds of this kind for other
windows are in~\cite{GebCap,Trmdy}. For every window of the regularity
above the limit is $0.68185$ (Corollary~\ref{cor:allwindows}), in agreement
with~\cite[Section~7.2]{AF}. Figure~\ref{fig:ladder} shows the numbers for
both quantities.

\subsection*{Organisation and formal verification}
Each bound is proved in four steps, shown in Figure~\ref{fig:method}: the
energy of the window (Proposition~\ref{prop:energy}); a local inequality for
seven or eight points (Propositions~\ref{prop:local}, \ref{prop:local2}
and~\ref{prop:local3}); the bound for the interaction energy of all separated
points that follows from it (Theorem~\ref{thm:gram},
Proposition~\ref{prop:gram2}); and the count of the zeros
(Theorems~\ref{thm:general}, \ref{thm:line-general}
and~\ref{thm:further-general}).
Section~\ref{sec:matrix} contains the inequalities, with all parameters
general. Section~\ref{sec:proof} proves Theorem~\ref{thm:main}: the operator,
the passage from a certificate to a bound, and then the constants.
Section~\ref{sec:cert} describes the verification of the seven-point
inequality, Section~\ref{sec:line} proves Theorem~\ref{thm:line},
Section~\ref{sec:further} proves Theorems~\ref{thm:further}
and~\ref{thm:third}, and
Section~\ref{sec:limits} describes the limits of the method. The proofs of
Theorems~\ref{thm:main}, \ref{thm:line}, \ref{thm:further}
and~\ref{thm:third} and of
Corollary~\ref{cor:simple} are formalised in Lean~4 on top of the library of~\cite{Claude26,AF}, which
derives the analytic inputs from the axioms of Mathlib and has the machinery
for a discretised form of~\eqref{eq:energy-intro} for a class of windows;
Section~\ref{sec:lean} states what is proved and what is trusted. The Lean
developments, the certificate data and the scripts accompany this
paper.\footnote{\url{https://github.com/kristianmk/839}; the same files,
except six settings files of code tools, are attached to the arXiv version of
this paper.} We refer to them as the
accompanying files.

\section{The inequalities}\label{sec:matrix}

Throughout this section $f$ is an even probability density on $I$, zero
outside $I$, and $K=\hat f$. Then $K$ is real and even, $K(0)=1$ and
$|K|\le1$. All other parameters are general; their values are fixed in
Section~\ref{sec:proof}.

For real $c$ let
\[
 \varphi_c(t)=t^2-(t-c)_+^2=\begin{cases}t^2,&t\le c,\\2ct-c^2,&t>c,\end{cases}
\]
a quadratic saturated to linear growth above the level $c$. It is convex and
$C^1$. For Hermitian $X$ write $\tr\varphi_c(X)=\sum_i\varphi_c(\lambda_i(X))$,
and $X\preceq Y$ when $Y-X$ is positive semidefinite.

For Hermitian $X$ let $X_\pm\succeq0$ carry its positive and negative
eigenvalues, so that $X=X_+-X_-$ and $X_+X_-=0$.

\begin{lemma}[{\cite[Lemma~4]{K1}}]\label{lem:threshold}
Let $P\succeq0$, let $Q$ be Hermitian with at most $p$ positive eigenvalues,
let $A=P+Q$, and let $c$ be real. Then
$\tr A^2\ge2c\tr A-2c\tr P+\tr\varphi_c(P)-c^2p$.
\end{lemma}

\begin{proof}
Write $Z=Q_-$, so that $A=(P-Z)+Q_+$. The cross term is
$2\tr((P-Z)Q_+)=2\tr(PQ_+)\ge0$, since $ZQ_+=0$. Hence
$\tr A^2\ge\tr(P-Z)^2+\tr Q_+^2$. Also $\tr Q_+^2\ge2c\tr Q_+-c^2p$, because
$x^2\ge2cx-c^2$. It remains to show $\tr(P-Z)^2+2c\tr Z\ge\tr\varphi_c(P)$.
Since $\varphi_c(t)=t^2-(t-c)_+^2$, this is
$\tr X_+^2-2\tr(XZ)+\tr Z^2\ge0$ with $X=P-cI$. Now
$\tr(XZ)\le\tr(X_+Z)$ because $\tr(X_-Z)\ge0$. So the left side is at least
$\tr(X_+-Z)^2\ge0$.
\end{proof}

For $c=2$ the lemma contains the rank--trace inequality
of~\cite[Lemma~3.2]{AF}, because $\varphi_2(t)\ge2t-1$ for $t>0$.

\begin{lemma}[pinching]\label{lem:pinch}
Let $\varphi$ be convex and let $G$ be Hermitian. For any partition of the
index set into blocks, $\tr\varphi(G)\ge\sum_B\tr\varphi(G_{BB})$.
\end{lemma}

\begin{proof}
Choose an orthonormal eigenbasis of each block $G_{BB}$. Together these form
an orthonormal basis $(e_a)$, and $\langle e_a,Ge_a\rangle$ runs over the
eigenvalues of the blocks. By Jensen's inequality for the spectral measure of
$G$ at $e_a$, $\langle e_a,\varphi(G)e_a\rangle\ge\varphi(\langle e_a,Ge_a\rangle)$.
Sum over $a$.
\end{proof}

A \emph{labelled set} $(Y,D)$ is a finite set $Y=\{y_1<\dots<y_l\}$ of real
numbers with labels $d_i\in\{1,2\}$ and $D=\operatorname{diag}(d_i)$. We write
$n_1,n_2$ for the numbers of points with label $1,2$ and
$\spn(Y)=y_l-y_1$ for the \emph{spread}, with $\spn(Y)=0$ if $Y$ is empty. For $\theta>0$ the set is \emph{$\theta$-separated} if
$|y_i-y_j|\ge\theta$ for $i\ne j$.

\begin{lemma}[separated points]\label{lem:maj}
Let $\theta>0$ and let $g$ be an integrable function with $g\ge f$ whose
Fourier transform vanishes outside $(-\theta,\theta)$, and put
$\Lambda=\int g$. If $w$ is measurable and even with $0\le w\le g$ and
$y_1,\dots,y_m$ are $\theta$-separated, then
$\bigl(\hat w(y_i-y_j)\bigr)\preceq\Lambda\cdot I$.
\end{lemma}

\begin{proof}
Let $v\in\mathbb{R}^m$ and $E(u)=\sum_jv_je^{2\pi iy_ju}$. Then
\[
 \sum_{i,j}v_iv_j\hat w(y_i-y_j)=\int w\,|E|^2\le\int g\,|E|^2
 =\sum_{i,j}v_iv_j\hat g(y_i-y_j)=\Lambda\sum_jv_j^2 ,
\]
because $\hat g(y_i-y_j)=0$ for $i\ne j$.
\end{proof}

\begin{lemma}[close pairs]\label{lem:pair}
Let $d,d'\in\{1,2\}$, let $\beta>0$ with $2+2\sqrt\beta\le c$, and let
$u^2\ge\beta$. Then
$G=\bigl(\begin{smallmatrix}d&\sqrt{dd'}\,u\\\sqrt{dd'}\,u&d'\end{smallmatrix}\bigr)$
satisfies $\tr\varphi_c(G)\ge d^2+d'^2+2dd'\beta$.
\end{lemma}

\begin{proof}
The eigenvalues $\lambda_\pm$ of $G$ have sum $d+d'$ and difference
$\sqrt{(d-d')^2+4dd'u^2}$. For fixed sum,
$\varphi_c(\lambda_+)+\varphi_c(\lambda_-)$ is nondecreasing in the
difference, because $\varphi_c$ is convex. So it is at least its value at
$u^2=\beta$. There $\lambda_+\le2+2\sqrt\beta\le c$, so both eigenvalues
are in the range where $\varphi_c(t)=t^2$, and the value is
$\tr G^2=d^2+d'^2+2dd'\beta$.
\end{proof}

\subsection*{The local inequality}
A \emph{frame} is a set of seven points $y_0<\dots<y_6$ with labels
$d_i\in\{1,2\}$ and gaps $g_r=y_{r+1}-y_r\ge\theta$, $0\le r\le5$.
\emph{Local data} for $K$, $\theta$ and a number $\alpha\ge0$ consist of
\begin{itemize}
\item weights $W^{dd'}_{ij}=W^{d'd}_{ij}\ge0$ for $0\le i<j\le6$ and
$d,d'\in\{1,2\}$, with
\begin{equation}\label{eq:budget}
 \sum_{i=0}^{6-s}W^{dd'}_{i,i+s}\le2dd'\qquad(1\le s\le6);
\end{equation}
\item rewards $0\le\kappa_{i,1}\le\kappa_{i,2}$, $0\le i\le6$, with sums
$x_1=\sum_i\kappa_{i,1}$ and $x_2=\sum_i\kappa_{i,2}$;
\item bounded functions $B^{dd'}_r=B^{d'd}_r$ of one variable, $0\le r\le5$,
with
\begin{equation}\label{eq:zero}
 \sum_{r=0}^5B^{dd'}_r(t)=0\qquad\text{for all }t\text{ and all }d,d' ;
\end{equation}
\end{itemize}
such that for every frame
\begin{equation}\label{eq:local}
 \alpha\,(y_6-y_0)+\sum_{0\le i<j\le6}W^{d_id_j}_{ij}K(y_j-y_i)^2
 +\sum_{r=0}^5B^{d_rd_{r+1}}_r(g_r)\ \ge\ \sum_{i=0}^6\kappa_{i,d_i}.
\end{equation}
In words: seven consecutive points cannot at the same time have a small spread
and small overlaps. The left side charges the spread at the rate $\alpha$
and the overlaps with weights that may grow with the multiplicities; the right
side is a reward, larger for double points. Inequalities of this form for
seven points without labels and with $B_r=0$ were introduced for simple zeros
in~\cite{Ainta}. Put
$B^{dd'}_{\le r}=\sum_{j\le r}B^{dd'}_j$ and
\[
 \Phi(g_0,\dots,g_4;d_0,\dots,d_5)=\sum_{r=0}^4B^{d_rd_{r+1}}_{\le r}(g_r),
\]
and let $\Omega$ be a bound for the oscillation $\sup\Phi-\inf\Phi$ over all
arguments.

\begin{lemma}[transfer]\label{lem:transfer}
Assume local data as above. Let $Y=\{y_1<\dots<y_m\}$, $m\ge0$, be
$\theta$-separated with
labels $d_i\in\{1,2\}$, let $\epsilon\ge0$, and let $e_{ij}=e_{ji}\ge0$
$(0<|i-j|\le6)$ be numbers with $e_{ij}\ge K(y_i-y_j)^2-\epsilon$. Then
\[
 \sum_{0<|i-j|\le6}d_id_j\,e_{ij}\ \ge\ x_1n_1+x_2n_2-6\alpha\,\spn(Y)-6x_2-\Omega-48\,\epsilon\,m .
\]
\end{lemma}

\begin{proof}
If $m\le6$ the right side is at most $mx_2-6x_2\le0$ and the left side is
nonnegative. Let $m\ge7$ and add
\eqref{eq:local} over the $m-6$ frames of consecutive points.

\emph{Pairs.} A pair of points at index distance $s\le6$ occurs in the frames
at the positions $(i,i+s)$, always with the same labels. By~\eqref{eq:budget}
its total weight is at most $2d_id_j\le8$, and $2d_id_j$ is its weight on the
left side of the lemma. Replacing $K^2$ by $e_{ij}$ therefore costs at most
$8\epsilon$ per pair, and there are at most $6m$ such pairs.

\emph{Spread.} Each gap lies in at most six frames, so the spread terms add up to
at most $6\alpha\,\spn(Y)$.

\emph{Rewards.} A point receives $\kappa_{i,d}$ for every position $i$ it takes
in a frame. For each position exactly six points are left out, so the rewards
add up to at least $x_1n_1+x_2n_2-6\sum_i\kappa_{i,2}=x_1n_1+x_2n_2-6x_2$.

\emph{Corrections.} By~\eqref{eq:zero}, $-B^{dd'}_{\le4}=B^{dd'}_5$, and
therefore
\[
 \sum_{r=0}^5B^{d_rd_{r+1}}_r(g_r)
 =\Phi(g_0,\dots,g_4;d_0,\dots,d_5)-\Phi(g_1,\dots,g_5;d_1,\dots,d_6).
\]
The second term is the first term of the next frame. So the corrections
telescope to a difference of two values of $\Phi$, which is at most $\Omega$.
\end{proof}

For points without labels, Lemma~\ref{lem:transfer} with a correction of this
kind is~\cite[(3.13)--(3.14)]{Chandru}, with the same function $\Phi$ and the
same loss $\Omega$. Labels of this kind were found independently by
D'Souza~\cite[Definition~5.1]{DSouza}, who calls them marks.

\begin{theorem}[labelled Gram bound]\label{thm:gram}
Assume local data for $K$, $\theta$ and $\alpha$, with reward sums $x_1,x_2$
and oscillation bound $\Omega$. Assume that $K(t)^2\ge\beta$ for
$|t|\le\theta$, where $\beta>0$ satisfies $2+2\sqrt\beta<c$ and
$2dd'\beta\ge x_d+x_{d'}$ for $d,d'\in\{1,2\}$. Then there are constants
$\eta_0>0$ and $C_1$ with the following property. Let $(Y,D)$ be a labelled
set with $l$ points, let $U$ be a real symmetric $l\times l$ matrix and
$0\le\eta\le\eta_0$ such that
\begin{enumerate}
\item[(a)] $|U_{ij}-K(y_i-y_j)|\le\eta$ for all $i,j$, and
\item[(b)] $U_{SS}\preceq\frac c2I$ for every $\theta$-separated subset $S$.
\end{enumerate}
Then $G=D^{1/2}UD^{1/2}$ satisfies
\[
 \tr\varphi_c(G)-\tr D^2\ \ge\ x_1n_1+x_2n_2-6\alpha\,\spn(Y)-(6x_2+\Omega)-C_1\eta\,l .
\]
\end{theorem}

\begin{proof}
Choose a maximal family of disjoint pairs $\{i,j\}$ with $|y_i-y_j|<\theta$.
The set $S$ of unmatched points is $\theta$-separated, since otherwise the
family could be enlarged. Apply Lemma~\ref{lem:pinch} to the partition into the
pairs and $S$.

A pair with labels $d,d'$ gives the block with diagonal $dU_{ii}$, $d'U_{jj}$
and off-diagonal entry $\sqrt{dd'}\,U_{ij}$, where $|U_{ii}-1|\le\eta$,
$|U_{jj}-1|\le\eta$ and $|U_{ij}|\ge\sqrt\beta-\eta$ by (a). As in
Lemma~\ref{lem:pair}, for a fixed diagonal the trace of $\varphi_c$ is
nondecreasing in $|U_{ij}|$. At $|U_{ij}|=\sqrt\beta-\eta$ the larger
eigenvalue is at most $2(1+\eta)+2\sqrt\beta<c$ when $\eta_0$ is small, so the
value there is the sum of the squares of the entries,
\begin{align*}
 d^2U_{ii}^2+d'^2U_{jj}^2+2dd'(\sqrt\beta-\eta)^2
 &\ge d^2+d'^2+2dd'\beta-16(1+\sqrt\beta)\eta\\
 &\ge d^2+d'^2+x_d+x_{d'}-16(1+\sqrt\beta)\eta .
\end{align*}

On $S$, hypothesis (b) gives $G_{SS}\preceq2\cdot\frac c2I=cI$, so no
eigenvalue is clipped and
\[
 \tr\varphi_c(G_{SS})=\tr G_{SS}^2=\sum_{i\in S}d_i^2U_{ii}^2+\sum_{i\ne j\in S}d_id_jU_{ij}^2 .
\]
The first sum is at least $\sum_{i\in S}d_i^2-8\eta|S|$. The second is at
least its part over the pairs at index distance at most six in $S$. To this
part apply Lemma~\ref{lem:transfer} with $e_{ij}=U_{ij}^2$ and
$\epsilon=3\eta$, which is admissible because $|K|\le1$: it is at least
$x_1n_1(S)+x_2n_2(S)-6\alpha\,\spn(S)-(6x_2+\Omega)-144\,\eta|S|$. Adding the
contributions and using $\spn(S)\le\spn(Y)$ gives the claim.
\end{proof}

For $U=(K(y_i-y_j))$ hypothesis (a) holds with $\eta=0$, and (b) follows from
Lemma~\ref{lem:maj} with $w=f$ when $\Lambda\le\frac c2$. The tolerance $\eta$
is needed because the operator below is built from a smoothed window.

\begin{remark}[the three devices in other terms]\label{rem:control}
Lemma~\ref{lem:maj} is the large sieve inequality in Selberg's form. It is the
direct half of an inequality of Ingham type~\cite{Ing36,KL05}: a gap condition
on the frequencies $y_j$ bounds $\int w\,|\sum_jv_je^{2\pi iy_ju}|^2$ by a
multiple of $\sum_j|v_j|^2$. The inverse half, which gives observability of
vibrating systems, needs a window longer than $1/\theta$ and is not used here.
Secondly, in~\cite[Lemma~3]{K1} the clipped trace was bounded
by a dual certificate, $\tr\varphi_c(G)\ge\tr(BG)-\frac14\tr B^2$ for every
$B\preceq2cI$. On a separated block this linear matrix inequality is inactive:
the optimal multiplier $B=2G$ is admissible and the bound is the identity
$\tr\varphi_c(G_{SS})=\tr G_{SS}^2$ used above, with nothing lost. The multiplier
chosen in~\cite[Theorem~2]{K1} is built from the matrix $(K(y_i-y_j))-I$ without
the labels, cut off at the level $\tau$; this is where the weights $d_id_j$
and the energy above the level $1+\tau$ were given up. Thirdly,
\eqref{eq:local} is a dissipation inequality in the sense of
Willems~\cite{Wil72}. Let the state of a frame be
$\sigma=(g_0,\dots,g_4;d_0,\dots,d_5)$ and let $\sigma'$ be the state of the next frame.
Then~\eqref{eq:local} says
\[
 \underbrace{\alpha\,(y_6-y_0)+\sum_{i<j}W^{d_id_j}_{ij}K(y_j-y_i)^2}_{\text{cost of the frame}}
 \ \ge\ \underbrace{\sum_i\kappa_{i,d_i}}_{\text{reward}}\ +\ \Phi(\sigma')-\Phi(\sigma),
\]
so $\Phi$ is a storage function for the supply rate ``cost minus reward'',
and Lemma~\ref{lem:transfer} is the summation along the line, which leaves a
boundary term. In another language the correction is a coboundary. The
certificate used in~\cite{K1} has $\Phi=0$.
\end{remark}

\section{Proof of Theorem~\ref{thm:main}}\label{sec:proof}

\emph{The operator.} The construction is that of~\cite[Section~2]{Lam}, in the
form of~\cite[Section~3]{K1}, whose notation we keep. Let $f\in L^2(I)$ be an even probability density on $I$ with
$\sqrt f$ smooth on the interior of $I$. Let $\chi_\varepsilon$ be smooth even
cutoffs with $0\le\chi_\varepsilon\le1$ that are supported inside $I$ and tend
to $\mathbf 1_I$, put $m_\varepsilon=\int f\chi_\varepsilon^2\to1$ and
$f_\varepsilon=f\chi_\varepsilon^2/m_\varepsilon=\eta_\varepsilon^2$. Here
$\eta_\varepsilon$ is real, even, smooth and supported inside $I$, and
$\int f_\varepsilon=1$. Let $K_\varepsilon=\hat f_\varepsilon$. Thus
\begin{equation}\label{eq:smooth}
 f_\varepsilon\le f/m_\varepsilon,\qquad
 d_\varepsilon:=\|f_\varepsilon-f\|_1\to0,\qquad
 |K_\varepsilon(t)-K(t)|\le d_\varepsilon\ \text{ for all real }t .
\end{equation}
Let $L=\log T$ and $z_\rho=i(\rho-\frac12)L/(2\pi)$ for $0<\gamma\le T$. Zeros
on the line give real $z_\rho$. A zero off the line pairs with $1-\bar\rho$,
which has the same multiplicity and the point $\bar z_\rho$. Put
$F_z(u)=\eta_\varepsilon(u)e^{-2\pi iuz}$. With inner products conjugate-linear
in the first slot, $\langle F_z,F_w\rangle=K_\varepsilon(w-\bar z)$. The
operator $A=\sum_z\nu_zF_z\langle F_{\bar z},\cdot\rangle$, where the sum is
over the points $z=z_\rho$ and $\nu_z$ is the multiplicity, acts on the
finite-dimensional span of the $F_z$. It is Hermitian, and
\[
 \tr A=\sum_z\nu_zK_\varepsilon(0)=N,\qquad
 \tr A^2=\sum_{z,w}\nu_z\nu_wK_\varepsilon(z-w)^2 ,
\]
using that $K_\varepsilon$ is even.

\begin{proposition}[{\cite[Lemma~5]{BGST}, \cite[Proposition~5]{K1}}]\label{prop:energy}
For fixed $\varepsilon$,
\begin{equation}\label{eq:energy}
 \tr A^2=(C(f_\varepsilon)+o(1))N\quad(T\to\infty),\qquad
 C(f_\varepsilon)\to C(f)\quad(\varepsilon\to0),
\end{equation}
where $C(g)=\int g^2+\iint|s-t|\,g(s)g(t)\,ds\,dt$.
\end{proposition}

\begin{proof}
The unconditional Montgomery theorem~\cite[Lemma~5]{BGST} applies to real,
even, integrable test functions supported in $[-1,1]$ and Lipschitz at $0$,
and it sums over all zeros, on or off the line, with multiplicity. It carries
the weight $w(\rho-\rho')=4/(4-(\rho-\rho')^2)$. With
$R_\varepsilon=f_\varepsilon*f_\varepsilon$,
\[
 \bigl(R_\varepsilon-R_\varepsilon''/(4L^2)\bigr)^{\wedge}(x)
 =\Bigl(1+\frac{\pi^2x^2}{L^2}\Bigr)K_\varepsilon(x)^2 .
\]
At $x=z_\rho-z_{\rho'}$ the factor is $1-(\rho-\rho')^2/4=1/w(\rho-\rho')$.
Both $R_\varepsilon$ and $R_\varepsilon''$ are admissible fixed test
functions, so applying the theorem to both removes the weight exactly;
in~\cite[Section~3]{Lam} the weight is removed by a modification of the test
function as well. The
term with $R_\varepsilon$ gives
$\bigl(R_\varepsilon(0)+\int|\alpha|R_\varepsilon(\alpha)\,d\alpha+o(1)\bigr)N
=(C(f_\varepsilon)+o(1))N$, and the term with $R_\varepsilon''$ contributes
$O_\varepsilon(N/L^2)$. Finally $f_\varepsilon\to f$ in $L^1$ and in $L^2$, so
$C(f_\varepsilon)\to C(f)$.
\end{proof}

\emph{Splitting.} Let $c\ge0$, and let $y_1<\dots<y_l$ be the rescaled heights of the distinct
zeros on the line of multiplicity $d_i\in\{1,2\}$, and let $n_1,n_2$ be the
numbers of each kind. Let $h$ count the zeros on the line of multiplicity at
least three and $k$ the pairs $\{\rho,1-\bar\rho\}$ off the line, so that
$N_d=n_1+n_2+h+2k$. Let $P$ be the part of $A$ coming from the $y_i$. Then
$P=VDV^*$ with $Ve_j=F_{y_j}$, and $U_\varepsilon=V^*V$ has the entries
$K_\varepsilon(y_i-y_j)$. So $P$ has the nonzero spectrum of
$G_\varepsilon=D^{1/2}U_\varepsilon D^{1/2}$. Since $\varphi_c(0)=0$, it
follows that $\tr\varphi_c(P)=\tr\varphi_c(G_\varepsilon)$. Also
$\tr P=n_1+2n_2$. In $Q=A-P$, each zero on the line of multiplicity at least
three gives a positive term of rank one. Each pair off the line gives
$\nu(vw^*+wv^*)=\frac\nu2\bigl((v+w)(v+w)^*-(v-w)(v-w)^*\bigr)$, with $v=F_z$
and $w=F_{\bar z}$, which has at most one positive eigenvalue. The number of
positive eigenvalues is subadditive, so $Q$ has at most $h+k$ of them. The
splitting and this count follow~\cite{AF,Lam}.
Lemma~\ref{lem:threshold} with $p=h+k$ gives
\begin{equation}\label{eq:key}
\begin{split}
 \tr A^2\ \ge\ {}&2cN-(2c-1)n_1-(4c-4)n_2-c^2(h+k)\\
 &+\bigl(\tr\varphi_c(G_\varepsilon)-\tr D^2\bigr).
\end{split}
\end{equation}
Everything now depends on a lower bound for the last term,
$\tr\varphi_c(G_\varepsilon)-\tr D^2$, which we call the \emph{clipped
interaction energy}.

\emph{Certificates.} A triple of real numbers $(x_1,x_2,b)$ with $b\ge0$ is a
\emph{certificate} for $f$ and $c$ if there are a constant $C_2$ and a
function $\omega(\varepsilon)\to0$ such that, for all small $\varepsilon$ and
all labelled sets with $l$ points,
\begin{equation}\label{eq:certificate}
 \tr\varphi_c(G_\varepsilon)-\tr D^2\ \ge\ x_1n_1+x_2n_2-b\,\spn(Y)-\omega(\varepsilon)\,l-C_2 .
\end{equation}

\begin{theorem}[certificate principle]\label{thm:general}
Let $(x_1,x_2,b)$ be a certificate for $f$ and $c$, and put
\begin{align*}
 \mu&=2-2x_1+x_2, & \nu&=2c-3+x_1-x_2,\\
 r_h&=6c-7-c^2+x_1-2x_2, & r_k&=4c-2-c^2-2x_1 .
\end{align*}
If $\mu>0$ and $\nu,r_h,r_k\ge0$, then
\[
 \liminf_{T\to\infty}\frac{N_d(T)}{N(T)}\ \ge\ \frac{3-C(f)+x_2-x_1-b}{2-2x_1+x_2}.
\]
\end{theorem}

\begin{proof}
Insert~\eqref{eq:certificate} into~\eqref{eq:key}:
\[
 \tr A^2\ \ge\ 2cN-(2c-1-x_1)n_1-(4c-4-x_2)n_2-c^2(h+k)-b\,\spn(Y)-\omega(\varepsilon)N-C_2.
\]
Note that $2c-1-x_1=\nu+\mu$ and $4c-4-x_2=2\nu+\mu$: a simple zero on the
line contributes $1$ to $N$ and $1$ to $N_d$, a double zero $2$ and $1$. A
zero on the line of multiplicity $j\ge3$ contributes $j$ to $N$ and $1$ to
$N_d$, and $j\nu+\mu-c^2\ge3\nu+\mu-c^2=r_h\ge0$. A pair off the line of
multiplicity $j\ge1$ contributes $2j$ and $2$, and
$2j\nu+2\mu-c^2\ge2\nu+2\mu-c^2=r_k\ge0$. Hence
\[
 (2c-1-x_1)n_1+(4c-4-x_2)n_2+c^2(h+k)\ \le\ \nu N+\mu N_d .
\]
The points $y_i$ lie in an interval of length $TL/2\pi=N+o(N)$. With
\eqref{eq:energy} this gives
\[
 \mu\,N_d\ \ge\ \bigl(3+x_2-x_1-b-C(f_\varepsilon)-\omega(\varepsilon)\bigr)N-o(N).
\]
Let $T\to\infty$ and then $\varepsilon\to0$.
\end{proof}

The bounds of~\cite{AF,Lam} and of~\cite{K1} are instances, and the bound
of~\cite{Chandru} has the value of the same formula.
\begin{enumerate}
\item \emph{No interaction.} The triple $(0,0,0)$ is a certificate for
every window and every $c\ge2$: by Lemma~\ref{lem:pinch} for the partition into
single points, $\tr\varphi_c(G_\varepsilon)\ge\sum_i\varphi_c(d_i)=\tr D^2$.
With $c=3$ the theorem gives $(3-C(f))/2$, which is $\frac56$ for the flat
window and $0.83625$ for the Montgomery--Taylor window~\cite{AF,Lam}.
\item \emph{Only simple points interact.} The bound $0.83652$
of~\cite[(1.9)]{Chandru}, for the Montgomery--Taylor window, is the value of
the formula of Theorem~\ref{thm:general} at
$(x_1,x_2,b)=(\frac{19}{5000},0,\frac1{500})$. It is proved there by a count
of its own, from a seven-point inequality for simple points.
\item \emph{A certificate that does not see the multiplicities.} In~\cite{K1},
Theorem~2 bounds $\tr\varphi_c(G)-\tr D^2$ below by a spectral function of the
unlabelled Gram matrix, clipped at a level $\tau$, and Section~3 there shows
that this function is at least $a\,l-\beta'\,\spn(Y)$ up to admissible errors.
This is the certificate $(a,a,\beta')$ with $a=\frac{26163}{5900000}$,
$\beta'=\frac{17}{7788}$ and $c=\frac{17}5$; here $a$ is the constant
of~\cite{K1} and $\beta'$ is its constant $\beta$, a letter that has another
meaning in this paper. So $(3-C-\beta')/(2-a)$, with
$C$ replaced by the upper bound for it used in~\cite{K1}, is the constant
of~\cite{K1}.
\item \emph{This paper.} Under the assumptions of Theorem~\ref{thm:gram} and
Lemma~\ref{lem:maj} with $\Lambda<\frac c2$, the triple $(x_1,x_2,6\alpha)$ is
a certificate. Indeed, apply Theorem~\ref{thm:gram} to $U=U_\varepsilon$.
Hypothesis (a) holds with $\eta=d_\varepsilon$ by~\eqref{eq:smooth}. For (b),
apply Lemma~\ref{lem:maj} to $w=m_\varepsilon f_\varepsilon\le f\le g$: on
separated sets $(U_\varepsilon)_{SS}\preceq(\Lambda/m_\varepsilon)I\preceq\frac c2I$
once $\varepsilon$ is small. So~\eqref{eq:certificate} holds with
$\omega(\varepsilon)=C_1d_\varepsilon$ and $C_2=6x_2+\Omega$.
\end{enumerate}

\emph{The window and the constants.} The \emph{first window} is
\[
 f(u)=1+\sum_{j=1}^6a_j\cos(2\pi ju)\qquad(|u|\le\tfrac12),
\]
with
\[
 10^6\,(a_1,\dots,a_6)=(193305,\,-23105,\,29166,\,13952,\,-21009,\,41052).
\]
Then $\int f=1$, $f\ge1-\sum|a_j|>0.678$ on $I$, and
\begin{equation}\label{eq:K}
 K(t)=\sinc t+\sum_{j=1}^6\frac{a_j}2\bigl(\sinc(t-j)+\sinc(t+j)\bigr),
 \qquad\sinc t=\frac{\sin\pi t}{\pi t},\quad\sinc0=1.
\end{equation}
Take $\theta=\frac{17}{20}$ and $c=\frac{67}{20}$.
\begin{itemize}
\item \emph{Majorant.} Let
$g(u)=\gamma_0\bigl(\sinc(\theta(u-\tfrac12))+\sinc(\theta(u+\tfrac12))\bigr)^2$
with $\gamma_0=\frac{607}{1000}$. The function
$\sinc(\theta(u-\frac12))+\sinc(\theta(u+\frac12))$ is the Fourier transform
of $\Psi(\xi)=\frac2\theta\cos(\pi\xi)\mathbf 1_{|\xi|\le\theta/2}$, so $g$ is
the transform of $\gamma_0\Psi*\Psi$, which is continuous and supported in
$[-\theta,\theta]$. By Fourier inversion $\hat g$ vanishes outside
$(-\theta,\theta)$, and
$\Lambda=\hat g(0)=\gamma_0\int\Psi^2=\frac{2\gamma_0}\theta(1+\sinc\theta)=1.67105\ldots<\frac c2=1.675$.
The inequality $f\le g$ on $I$ has margin at least $10^{-4}$ and is proved by
rational interval arithmetic on a subdivision of $[0,\frac12]$.
\item \emph{Close pairs.} $K(t)^2\ge\beta:=\frac{8824337}{134217728}=0.06574\ldots$
for $|t|\le\theta$, from the kernel table of Section~\ref{sec:cert} and
$|K'|\le\pi$. Moreover $2+2\sqrt\beta=2.5129<c$.
\item \emph{Local data.} The rational data of Appendix~\ref{app:data} have
\[
 \alpha=0.006219,\qquad x_1=0.054350,\qquad x_2=0.069433,\qquad\Omega=0.015882 ,
\]
they satisfy~\eqref{eq:budget} and~\eqref{eq:zero}, and the functions $B_r$
are piecewise linear and supported in $[\frac34,\frac{19}8]$. The value of
$\Omega$ is the sum over $r\le4$ of the ranges of $B^{dd'}_{\le r}$ over $t$ and
the three pairs of labels, read off from the knot values. Then
$2\beta>2x_1$, $4\beta>x_1+x_2$ and $8\beta>2x_2$.
\end{itemize}

\begin{proposition}[local inequality; computer-assisted]\label{prop:local}
The data of Appendix~\ref{app:data} are local data for $K$,
$\theta=\frac{17}{20}$ and $\alpha=0.006219$: inequality~\eqref{eq:local}
holds for every frame.
\end{proposition}

The proof is in Section~\ref{sec:cert}.

\emph{Conclusion.} By the fourth instance above, $(x_1,x_2,b)$ with
$b=6\alpha=0.037314$ is a certificate for $f$ and $c$. The conditions of
Theorem~\ref{thm:general} hold: $\mu=1.960733$, $\nu=3.684917$,
$r_h=1.792984$ and $r_k=0.0688$. Direct integration gives
\begin{equation}\label{eq:C}
 C(f)=\frac43+\frac12\sum_ja_j^2+\frac1{\pi^2}\sum_j\Bigl(\frac{(-1)^ja_j}{j^2}-\frac{a_j^2}{4j^2}\Bigr)
 =1.3327049582\ldots<C_0:=\frac{266541}{200000},
\end{equation}
and therefore
\[
 \liminf_{T\to\infty}\frac{N_d(T)}{N(T)}\ \ge\ \frac{3-C_0+x_2-x_1-b}{2-2x_1+x_2}=\frac{1645064}{1960733}. \qed
\]

\begin{remark}\label{rem:c}
The clipping level is confined to a short interval. The step for pairs off
the line needs $r_k\ge0$, that is $c\le2+\sqrt{2-2x_1}=3.3752$, and the
absence of clipping needs $c>2\Lambda=3.3421$. The value of $g$ at the
endpoints is $\gamma_0(1+\sinc\theta)^2=0.8309$, and $f(\frac12)=0.8304$: here
$\Lambda$ is governed by the height of the window at its edge.

The separation $\theta$ is confined in the same way. Two simple points closer
than $\theta$ are not covered by the local inequality and must pay their
rewards by their own overlap, which needs $K(\theta)^2\ge x_1$; this fails for
$\theta>0.8705$. On the other side, $f(\frac12)\le g(\frac12)$ gives
$\Lambda\ge2f(\frac12)/(\theta(1+\sinc\theta))$ for a majorant of the form
used here, and $\Lambda<\frac c2$ then fails for $\theta<0.8267$ when
$c=\frac{67}{20}$, and for $\theta<0.7855$ whatever admissible $c$ is chosen.
The value $\frac{17}{20}$ is a rational number between these limits. The choice matters
little for the constant. The trial states of Section~\ref{sec:six} have all
gaps larger than $1$, so the ceilings of Proposition~\ref{prop:six} hold for
every separation up to $1$: with the first window, no choice of $\theta$ gives
more
than $0.83916$ with frames of seven points.
\end{remark}

\section{The seven-point inequality}\label{sec:cert}

Proposition~\ref{prop:local} is an inequality in six real variables for each
of the $2^7$ labellings. Its proof is a finite search. The proof of record is
the Lean development of Section~\ref{sec:lean}, in which the search is a
function with a soundness theorem; here we describe what it does. The plan of
the search, with a table of lower bounds on grid cells, range minima, a
cutoff for long frames and bisection of the widest gap, is that of the
verifier of~\cite{Ainta}.

\emph{Reduction to a compact set.} The functions $B_r$ are bounded below. The
sum over $r$ of the largest value of $-B^{dd'}_r(t)$ over $t$ and the labels
is $0.009902$. Since $K^2\ge0$, inequality~\eqref{eq:local} holds whenever
$\alpha\sum g_r\ge x_2+0.009902$, in particular when $\sum g_r\ge k_{\max}/H$
with $H=16384$ and $k_{\max}=209009$.

\emph{Kernel table.} For $0\le k\le k_{\max}$ an integer $v_k$ is computed
with $K(t)^2\ge v_k/2^{28}$ on the cell $[k/H,(k+1)/H]$. From~\eqref{eq:K},
$K(t)=\frac{\sin\pi t}{\pi}\bigl(\frac1t+\sum_j(-1)^ja_j\frac t{t^2-j^2}\bigr)$
for non-integral $t$. At the midpoint of the cell, a lower bound for $|K|$ and
an upper bound for $|K'|$ are obtained from Taylor polynomials with rational
arguments and rational bounds for $\pi$. Since $f\ge0$, $\int f=1$ and
$|u|\le\frac12$, one has $|K'|\le\pi$ and $|K''|\le\pi^2$, which controls $K$
on the whole cell.

\emph{Branch and bound.} A box is a product of six intervals of gaps with
endpoints in $\frac1H\mathbb Z$. The search starts from the box
$[13926/H,k_{\max}/H]^6$, where $13926/H<\theta$. On a box each $K^2$ is
bounded below by the minimum of the table over the cells that the
corresponding distance can reach, and the piecewise linear functions are
minimised exactly. Two tests are applied.

The first treats all labellings at once. For a pair $e=(i,j)$ let
$E^{dd'}_e$ be its term $W^{dd'}_{ij}K(y_j-y_i)^2$, plus $B^{dd'}_i(g_i)$ when
$j=i+1$. Let $\ell_e\le E^{11}_e$, $u_e\le E^{12}_e-E^{11}_e$ and
$v_e\le E^{22}_e-E^{11}_e$ be lower bounds on the box, which the table gives
because $W^{11}\le W^{12}$ and $W^{11}\le W^{22}$, and put
$\eta_e=\min(2u_e,v_e)$. With $d_i=1+\epsilon_i$, the three cases give
$2E^{d_id_j}_e\ge2\ell_e+\eta_e(\epsilon_i+\epsilon_j)$. So twice the
difference of the two sides of~\eqref{eq:local} is at least
\[
 2\Bigl(\alpha\sum_rg_r+\sum_e\ell_e-\sum_i\kappa_{i,1}\Bigr)
 +\sum_i\min\Bigl(0,\ \sum_{e\ni i}\eta_e-2(\kappa_{i,2}-\kappa_{i,1})\Bigr),
\]
where $\alpha\sum_rg_r$ is taken at the lower corner of the box, and the box is
accepted if this is nonnegative. Before the tests each box is cut down to
$\sum_rg_r\le k_{\max}/H$, which is where the reduction to a compact set
enters. For boxes of side at most
$512/H$ a second test replaces the minima by affine lower bounds for the
table and for the $B_r$, minimised at a corner of the box, first in the same
combined form and then for each labelling. A box that passes no test is
bisected along its longest side; a box of side $1/H$ that passes no test
would be a failure. The search ends after $60{,}467{,}309$ boxes, of which
$30{,}233{,}655$ are accepted, with no failure.

The coefficients were found by a floating-point search and then frozen as
integers; the search plays no role in the proof. The smallest slack
in~\eqref{eq:local} found numerically is about $10^{-4}$, for seven simple
points with gaps near $1.025$.

\section{Zeros on the critical line}\label{sec:line}

\emph{Why the level changes.} In~\eqref{eq:key} a zero on the line of
multiplicity $j\ge2$ that is put into the remainder is credited with
$2cj-c^2$, and a pair of zeros off the line with $4c-c^2$. For a count of the
zeros that are simple and on the line, each of these must be at least $2$ per
unit of multiplicity. For the double zero and for the pair this is
$4c-c^2\ge4$, that is, $(c-2)^2\le0$. So $c=2$. Theorem~\ref{thm:main} can
use a larger level, because zeros off the line are distinct, and for
$c=\frac{67}{20}$ a pair is credited with $2.18$, a little more than the
energy $2$ of two simple zeros. The price of the level $2$ is that every block
that contains a double point and a second point with nonzero overlap has an
eigenvalue larger than $2$ and is clipped. At the level $2$ we therefore let
only the simple zeros on the line interact. Their interaction is treated as in
Section~\ref{sec:matrix}.

\emph{Splitting.} Let now $y_1<\dots<y_n$ be the rescaled heights of the zeros
on the line that are simple, so that $n=N_0^s(T)$, let $P$ be the part of $A$
coming from them and $Q=A-P$. The nonzero spectrum of $P$ is that of
$U_\varepsilon=(K_\varepsilon(y_i-y_j))$. The operator $Q$ is the sum of the
terms of the zeros on the line of multiplicity $j\ge2$, which are positive of
rank one with trace $j$, and of the terms of the pairs off the line, which
have at most one positive eigenvalue and trace $2j$, as in
Section~\ref{sec:proof}. So $Q$ has at most $p$ positive eigenvalues, where
$p$ is the number of these zeros and pairs, and $N-n=\tr Q\ge2p$. Lemma~\ref{lem:threshold} with $c=2$
gives
\begin{equation}\label{eq:key2}
 \tr A^2\ \ge\ 4(N-n)-4p+\tr\varphi_2(U_\varepsilon)\ \ge\ 2(N-n)+\tr\varphi_2(U_\varepsilon).
\end{equation}
A pair of real numbers $(x,b)$ with $x<1$ and $b\ge0$ is a \emph{certificate
for the critical line} if there are a constant $C_2$ and a function
$\omega(\varepsilon)\to0$ such that, for all small $\varepsilon$ and all finite
sets $Y$ of $n$ real points,
\begin{equation}\label{eq:certificate2}
 \tr\varphi_2(U_\varepsilon)-n\ \ge\ x\,n-b\,\spn(Y)-\omega(\varepsilon)\,n-C_2 .
\end{equation}

\begin{theorem}[certificate principle for the critical line]\label{thm:line-general}
If $(x,b)$ is a certificate for the critical line, then
\[
 \liminf_{T\to\infty}\frac{N_0^s(T)}{N(T)}\ \ge\ \frac{2-C(f)-b}{1-x}.
\]
\end{theorem}

\begin{proof}
By \eqref{eq:key2}, \eqref{eq:certificate2}, \eqref{eq:energy} and
$\spn(Y)\le N+o(N)$,
\[
 (C(f_\varepsilon)+o(1))N\ \ge\ 2(N-n)+(1+x)n-bN-\omega(\varepsilon)N-C_2,
\]
that is, $(1-x)\,n\ge(2-C(f_\varepsilon)-b-\omega(\varepsilon))N-o(N)$. Let
$T\to\infty$ and then $\varepsilon\to0$.
\end{proof}

The pair $(0,0)$ is a certificate for every window, by
Lemma~\ref{lem:pinch} for the partition into single points. It gives
$2-C(f)$, which is $\frac23$ for the flat window and $0.67250$ for the
Montgomery--Taylor window~\cite{AF,Lam}. The count of
Theorem~\ref{thm:line-general} is used in~\cite{Ainta} and
in~\cite[(1.8)]{Chandru}, for the Montgomery--Taylor window. The bound
$0.67305$ of~\cite{Chandru} is its value at
$(x,b)=(\frac{19}{5000},\frac1{500})$, and the constant of~\cite{Ainta}, as
quoted in~\cite{Chandru}, is its value at
$(x,b)=(\frac{4997}{1345000},\frac{2680}{1345000})$.

\emph{Frames of eight points.} The next statement is Theorem~\ref{thm:gram}
for simple points at the level $2$. The number seven has no part in the
arguments of Section~\ref{sec:matrix}: a longer frame sees more pairs, and its
local inequality needs a longer search. For the critical line we use
\emph{frames of eight points} $y_0<\dots<y_7$ with gaps
$g_r=y_{r+1}-y_r\ge\theta$, $0\le r\le6$, and no labels. \emph{Local data for
eight points}, for $K$, $\theta$ and a number $\alpha\ge0$, consist of weights
$W_{ij}\ge0$ for $0\le i<j\le7$ with $\sum_{i=0}^{7-s}W_{i,i+s}\le2$ for
$1\le s\le7$, rewards $\kappa_i\ge0$ with sum $x$, and bounded functions $B_r$
of one variable, $0\le r\le6$, with $\sum_{r=0}^6B_r=0$, such that for every
frame of eight points
\begin{equation}\label{eq:local8}
 \alpha\,(y_7-y_0)+\sum_{0\le i<j\le7}W_{ij}K(y_j-y_i)^2+\sum_{r=0}^6B_r(g_r)\ \ge\ x .
\end{equation}
Put $B_{\le r}=\sum_{j\le r}B_j$ and
$\Phi(g_0,\dots,g_5)=\sum_{r=0}^5B_{\le r}(g_r)$, and let $\Omega$ be a bound
for the oscillation of $\Phi$. Let $y_1<\dots<y_m$ be $\theta$-separated, let
$\epsilon\ge0$, and let $e_{ij}=e_{ji}\ge0$ $(0<|i-j|\le7)$ be numbers with
$e_{ij}\ge K(y_i-y_j)^2-\epsilon$. Then
\begin{equation}\label{eq:transfer8}
 \sum_{0<|i-j|\le7}e_{ij}\ \ge\ x\,m-7\alpha\,(y_m-y_1)-7x-\Omega-14\,\epsilon\,m .
\end{equation}
The proof is that of Lemma~\ref{lem:transfer}, with the $m-7$ frames of eight
consecutive points when $m\ge8$: a pair at index distance $s\le7$ has total
weight at most $2$, which is its weight on the left side, and there are at
most $7m$ such pairs; each gap lies in at most seven frames; for each position
exactly seven points are left out; and since $-B_{\le5}=B_6$, the corrections
of a frame are $\Phi(g_0,\dots,g_5)-\Phi(g_1,\dots,g_6)$ and telescope.

\begin{proposition}[Gram bound for simple points]\label{prop:gram2}
Assume local data for eight points, for $K$, $\theta$ and $\alpha$, with
reward sum $x$ and oscillation bound $\Omega$. Assume that
$K(t)^2\ge\beta\ge x$ for $|t|\le\theta$. Then there are
$\eta_0>0$ and $C_1$ with the following property. Let $Y$ be a set of $n$ real
points, let $U$ be a real symmetric $n\times n$ matrix with $U_{ii}=1$ and
$|U_{ij}|\le1$, and let $0\le\eta\le\eta_0$ be such that
\begin{enumerate}
\item[(a)] $|U_{ij}-K(y_i-y_j)|\le\eta$ for all $i,j$, and
\item[(b)] $U_{SS}\preceq2I$ for every $\theta$-separated subset $S$.
\end{enumerate}
Then $\tr\varphi_2(U)-n\ \ge\ x\,n-7\alpha\,\spn(Y)-(7x+\Omega)-C_1\eta\,n$.
\end{proposition}

\begin{proof}
Choose a maximal family of disjoint pairs of points of $Y$ at distance less
than $\theta$. The points in no pair form a $\theta$-separated set $S$. By
Lemma~\ref{lem:pinch}, $\tr\varphi_2(U)$ is at least the sum of
$\tr\varphi_2(U_{SS})$ and of the terms of the pairs. On $S$ hypothesis~(b)
says that no eigenvalue exceeds $2$, so $\varphi_2$ is the square there. The
sum of the $U_{ij}^2$ over $i\ne j$ in $S$ is at least its part over the pairs
at index distance at most seven in $S$, and by~\eqref{eq:transfer8} with
$e_{ij}=U_{ij}^2$ and $\epsilon=3\eta$,
\[
 \tr\varphi_2(U_{SS})=|S|+\sum_{i\ne j\in S}U_{ij}^2\ \ge\ |S|+x\,|S|-7\alpha\,\spn(S)-7x-\Omega-42\eta\,|S| .
\]
A pair has the matrix $\bigl(\begin{smallmatrix}1&u\\u&1\end{smallmatrix}\bigr)$
with $|u|\le1$. Its eigenvalues $1\pm u$ are at most $2$, so its term is
$2+2u^2\ge2+2\beta-6\eta\ge2+2x-6\eta$. Add up.
\end{proof}

Under the assumptions of Proposition~\ref{prop:gram2}, and of
Lemma~\ref{lem:maj} with $\Lambda<2$, the pair $(x,7\alpha)$ is therefore a
certificate for the critical line: apply the proposition to $U=U_\varepsilon$,
with $\eta=d_\varepsilon$ in~(a) and with (b) from Lemma~\ref{lem:maj} for
$w=m_\varepsilon f_\varepsilon\le g$. The condition on the majorant is now
$\Lambda<c$ and not $\Lambda<\frac c2$, because no double point is in a block,
and no condition like $2+2\sqrt\beta<c$ is needed, because a pair of simple
points is never clipped at the level $2$.

\emph{The window and the constants.} For Theorem~\ref{thm:line} we take the
\emph{second window},
\begin{gather*}
 f(u)=1+\sum_{j=1}^{10}a'_j\cos(2\pi ju),\\
 \begin{aligned}
 10^6\,(a'_1,\dots,a'_{10})=(&122290,\,-45012,\,16241,\,1887,\,2711,\\
 &3731,\,-5307,\,2982,\,-5153,\,5358).
 \end{aligned}
\end{gather*}
Here $\int f=1$ and $f\ge1-\sum_j|a'_j|>0.789$ on $I$. By~\eqref{eq:C} with ten
terms, $C(f)=1.32794727\ldots<C'_0=\frac{6639737}{5000000}$, so the plain bound is
$2-C(f)=0.67205\ldots$, smaller than that of the Montgomery--Taylor window.
The separation is $\theta=\frac{17}{20}$ as before, and the majorant has the
form used in Section~\ref{sec:proof} with $\gamma_0=\frac{13}{20}$, so that
$\Lambda=1.78942\ldots<2$. The inequality $f\le g$ on $I$ has margin at
least $10^{-4}$ and is proved as in Section~\ref{sec:proof}. On
$[-\theta,\theta]$ we have $K^2\ge\beta=13707941/2^{28}=0.05106\ldots$, the
smallest entry of the kernel table on this interval. The
local data for eight points are in Tables~\ref{tab:line} and~\ref{tab:lineknots} of the appendix: $\alpha=627\cdot10^{-6}$,
the single reward $\kappa_3=x=8722\cdot10^{-6}\le\beta$, and
$\Omega=3882\cdot10^{-6}$.

\begin{proposition}[local inequality for simple points; computer-assisted]\label{prop:local2}
With these data, inequality~\eqref{eq:local8} holds for every frame of eight
points.
\end{proposition}

The proof is the computation of Section~\ref{sec:cert} with boxes of seven
intervals of gaps in place of six, on a new table of lower bounds for $K^2$,
for the kernel of the second window. The tests applied to a box are the same;
the data are used for all labels, so the inequality that is checked is the
inequality for eight points without labels. The table has $262{,}144$ entries
and extends to $t=16$ instead of $12.76$. The reason is that the charge
$\alpha$ for the spread is small: a frame is trivial only when
$\alpha\,(y_7-y_0)\ge x+0.001294$, that is, when its spread is at least
$15.975$. Here $0.001294$ is the sum over $r$ of the largest value of $-B_r$.
The search visits $535{,}332{,}163$ boxes and accepts $267{,}666{,}082$ of them,
with no failure. The reference programs and their data are in the accompanying
files, and Section~\ref{sec:lean} describes the formal proof.

\begin{proof}[Proof of Theorem~\ref{thm:line}]
By Propositions~\ref{prop:local2} and~\ref{prop:gram2} the pair
$(x,b)=(8722,\,4389)\cdot10^{-6}$ is a certificate for the critical line, and
by Theorem~\ref{thm:line-general}
\[
 \liminf_{T\to\infty}\frac{N_0^s(T)}{N(T)}\ \ge\ \frac{2-C'_0-b}{1-x}=\frac{1669159}{2478195}. \qedhere
\]
\end{proof}

\begin{remark}[why the gain is smaller]
The gain over $2-C(f)$ is $0.0015$, against $0.0054$ for distinct zeros.
At the level $2$ the double zeros do not interact, and the simple zeros have
density about $\frac23$. They can avoid each other: a configuration in which
gaps near $1$ and near $2$ alternate has all distances close to integers,
where $K$ is small. The window has to make $K$ at the integers larger, which
costs in $C(f)$, and this is the balance behind the coefficients $a'_j$.
\end{remark}

\section{Two further counts}\label{sec:further}
The two quantities of Theorem~\ref{thm:further} are bounded by the same
inequalities as before. What changes is the count of the zeros that are not
in the Gram part, and for the second quantity the level at which their energy
is estimated. The preprint~\cite{Chandru} also bounds the zeros of
multiplicity at most two, by $0.83736$ of $N(T)$ and by $0.93919$ of
$N_d(T)$; we do not treat these two counts.

\begin{lemma}[threshold with two levels]\label{lem:threshold2}
Let $P\succeq0$, let $Q$ be Hermitian with at most $p$ positive eigenvalues,
let $A=P+Q$, and let $t\ge c$ be real. Then
$\tr A^2\ge\tr\varphi_c(P)+2t\,(\tr A-\tr P)-t^2p$.
\end{lemma}

\begin{proof}
In the proof of Lemma~\ref{lem:threshold}, with $Z=Q_-$, we have
$\tr A^2\ge\tr(P-Z)^2+\tr Q_+^2$ and $\tr(P-Z)^2\ge\tr\varphi_c(P)-2c\tr Z$.
Now $\tr Q_+^2\ge2t\tr Q_+-t^2p$ and $\tr Q_+=\tr Q+\tr Z$. Hence
$\tr A^2\ge\tr\varphi_c(P)+2t\tr Q+(2t-2c)\tr Z-t^2p$, and $\tr Z\ge0$.
\end{proof}

For $t=c$ this is Lemma~\ref{lem:threshold}. The level $c$ decides how much of
the interaction inside $P$ is kept. The level $t$ sets the price of the zeros
in $Q$: a term of $Q$ with one positive eigenvalue and trace $u$ is credited
with the energy $2tu-t^2$, the tangent at $t$ to its least possible energy
$u^2$.

\begin{theorem}[two further counts]\label{thm:further-general}
\begin{enumerate}
\item[(a)] If $(x,b)$ is a certificate for the critical line with $x\ge0$,
then
\[
 \liminf_{T\to\infty}\frac{N_s(T)+N_0(T)}{2N(T)}\ \ge\ \frac{1+s}2,
 \qquad s=\frac{2-C(f)-b}{1-x}.
\]
\item[(b)] Let $(x_1,x_2,b)$ be a certificate for $f$ and $c$ with
$2-2x_1+x_2\ge0$, let $t\ge c$ be such that $4t-t^2\ge2+2x_1$, and put
$e_0=2t-\frac{t^2}4$. If $e_0>1+x_1$, then
\[
 \liminf_{T\to\infty}\frac{N_{s\cup0}(T)}{N(T)}\ \ge\ \frac{e_0-C(f)-b}{e_0-1-x_1}.
\]
\item[(c)] Let $(x,b)$ be a certificate for the critical line with $x\ge0$,
let $t\ge2$ be such that $4t-t^2\ge2+2x$, and put $e_0=2t-\frac{t^2}4$. If
$e_0>1+x$, then
\[
 \liminf_{T\to\infty}\frac{N_{s\cup0}(T)}{N(T)}\ \ge\ \frac{e_0-C(f)-b}{e_0-1-x}.
\]
\end{enumerate}
\end{theorem}

\begin{proof}
(a) We use the splitting of Section~\ref{sec:line}. Let $r$ be the number of
zeros on the line of multiplicity at least two and $D$ the sum of their
multiplicities. Let $k_1$ be the number of pairs off the line of multiplicity
one, $k_2$ the number of pairs off the line of multiplicity at least two, and
$M_2$ the sum of the multiplicities of the latter. Then $D\ge2r$,
$M_2\ge2k_2$,
\[
 N=n+D+2k_1+2M_2,\qquad N_s=n+2k_1,\qquad N_0=n+D,
\]
and $Q$ has at most $p=r+k_1+k_2$ positive eigenvalues. The first inequality
of~\eqref{eq:key2}, \eqref{eq:certificate2} and~\eqref{eq:energy} give
\[
 (C(f_\varepsilon)+o(1))N\ \ge\ 4(N-n)-4(r+k_1+k_2)+(1+x)n-bN-\omega(\varepsilon)N-C_2 .
\]
Here $4(r+k_1+k_2)\le2D+4k_1+2M_2$ and $2M_2=N-n-D-2k_1$, so the right side
is at least $3N-(2-x)n-D-2k_1-bN-\omega(\varepsilon)N-C_2$. Therefore
\[
 N_s+N_0=2n+D+2k_1\ \ge\ \bigl(3-C(f_\varepsilon)-b-\omega(\varepsilon)\bigr)N+x\,n-o(N),
\]
and $n\ge(s-o(1))N$ by Theorem~\ref{thm:line-general}, after $T\to\infty$ and
$\varepsilon\to0$. Finally $3-C(f)-b+xs=1+s$.

(b) We use the splitting of Section~\ref{sec:proof}. Let $H$ be the sum of the
multiplicities of the $h$ zeros on the line of multiplicity at least three,
and let $k_1$, $k_2$ and $M_2$ be as in~(a), so that $k=k_1+k_2$,
\[
 N=n_1+2n_2+H+2k_1+2M_2,\qquad N_{s\cup0}=N-2M_2,
\]
$H\ge3h$ and $M_2\ge2k_2$. By Lemma~\ref{lem:threshold2} with $p=h+k$,
by~\eqref{eq:certificate} and by $\tr P=n_1+2n_2$,
\begin{align*}
 \tr A^2&+b\,\spn(Y)+\omega(\varepsilon)N+C_2\\
 &\ge(1+x_1)n_1+(4+x_2)n_2+(2tH-t^2h)+(4t-t^2)k_1+(4tM_2-t^2k_2).
\end{align*}
Each term on the right is at least $1+x_1$ times its share of the
multiplicity, except the last: $4+x_2\ge2(1+x_1)$ by assumption;
$2tH-t^2h\ge(2t-\frac{t^2}3)H$ and
$2t-\frac{t^2}3\ge\frac12(4t-t^2)\ge1+x_1$; and $4t-t^2\ge2(1+x_1)$. The last
term is at least $(4t-\frac{t^2}2)M_2=e_0\cdot2M_2$. With~\eqref{eq:energy}
and $\spn(Y)\le N+o(N)$,
\[
 \bigl(C(f_\varepsilon)+b+\omega(\varepsilon)+o(1)\bigr)N\ \ge\ (1+x_1)(N-2M_2)+e_0\cdot2M_2 ,
\]
so $(e_0-1-x_1)\,2M_2\le\bigl(C(f_\varepsilon)+b+\omega(\varepsilon)-1-x_1+o(1)\bigr)N$,
and $N_{s\cup0}=N-2M_2$. Let $T\to\infty$ and then $\varepsilon\to0$.

(c) We use the splitting of Section~\ref{sec:line} and the notation of~(a).
Now $N_{s\cup0}=n+D+2k_1=N-2M_2$. By Lemma~\ref{lem:threshold2} with $c=2$
and $p=r+k_1+k_2$, by~\eqref{eq:certificate2} and by $\tr P=n$,
\[
 \tr A^2+b\,\spn(Y)+\omega(\varepsilon)N+C_2
 \ \ge\ (1+x)n+(2tD-t^2r)+(4t-t^2)k_1+(4tM_2-t^2k_2).
\]
Here $2tD-t^2r\ge(2t-\frac{t^2}2)D\ge(1+x)D$, because $D\ge2r$ and
$4t-t^2\ge2+2x$, and $(4t-t^2)k_1\ge(1+x)\,2k_1$ for the same reason. The last
term is at least $e_0\cdot2M_2$, as in~(b). Hence
$\bigl(C(f_\varepsilon)+b+\omega(\varepsilon)+o(1)\bigr)N\ge(1+x)(N-2M_2)+e_0\cdot2M_2$,
and the proof ends as in~(b).
\end{proof}

The bounds of~\cite{Lam} are the case without interaction. With the pair
$(0,0)$, part~(a) gives $(3-C(f))/2$, which is $0.83625$ for the
Montgomery--Taylor window. With the triple $(0,0,0)$, $c=2$ and
$t=2+\sqrt2$, part~(b) gives
$(5+2\sqrt2-2C(f))/(3+2\sqrt2)$, which is $0.88762$ for that
window~\cite[(2.7)]{Lam}. For $t=2+\sqrt{2-2x_1}$, the largest level with
$4t-t^2\ge2+2x_1$, the bound of part~(b) is
\[
 1-\frac{2\,(C(f)+b-1-x_1)}{3-x_1+2\sqrt{2-2x_1}} .
\]
This is the formula of~\cite[(1.10)]{Chandru}, where it is proved for a
certificate for simple points, as in part~(c); the value $0.88855$ stated
there is its value
at $(x_1,b)=(\frac7{400},\frac3{250})$ for a window with $C(f)=1.32792$.
The choice of a level adapted to the count is made there as well. What is
new in part~(b) is the certificate with labels, and in
Theorems~\ref{thm:further} and~\ref{thm:third} the values.

\begin{proof}[Proof of Theorem~\ref{thm:further}]
For the first bound take the certificate $(x,b)=(8722,\,4389)\cdot10^{-6}$ of
the second window in part~(a): $s\ge s_0$, because $C(f)<C'_0$. For the second bound take the
certificate $(x_1,x_2,b)=(54350,\,69433,\,37314)\cdot10^{-6}$ of the first
window, $c=\frac{67}{20}$ and $t=\frac{27}8$ in part~(b). Then $t\ge c$,
$4t-t^2=2.109375\ge2.1087=2+2x_1$ and $e_0=\frac{999}{256}$, and with the
bound $C_0=1.332705$ for $C(f)$ the quotient is $\frac{3376433}{3797325}$.
\end{proof}

\emph{The third window.} For Theorem~\ref{thm:third} we take the \emph{third
window},
\begin{gather*}
 f(u)=1+\sum_{j=1}^6a''_j\cos(2\pi ju),\\
 10^6\,(a''_1,\dots,a''_6)=(200566,\,19,\,-9358,\,9283,\,-6746,\,8170).
\end{gather*}
Here $\int f=1$ and $f\ge1-\sum_j|a''_j|>0.765$ on $I$. By~\eqref{eq:C},
$C(f)=1.33246363\ldots<C''_0=\frac{6662319}{5000000}$. The separation is
$\theta=\frac{17}{20}$ and the clipping level is $2$, as in
Section~\ref{sec:line}. The majorant has the form used in
Section~\ref{sec:proof} with $\gamma_0=\frac{61}{100}$, so that
$\Lambda=1.67931\ldots<2$, and the inequality $f\le g$ on $I$ has margin at
least $10^{-4}$. On $[-\theta,\theta]$ we have
$K^2\ge\beta=4483379/2^{26}=0.06680\ldots$, the smallest entry of the kernel
table on this interval. The local data are in Tables~\ref{tab:third}
and~\ref{tab:thirdknots} of the appendix: $\alpha=6212\cdot10^{-6}$, the
single reward $\kappa_3=x=54525\cdot10^{-6}\le\beta$, and
$\Omega=75420\cdot10^{-6}$.

\begin{proposition}[local inequality for the third window; computer-assisted]\label{prop:local3}
With these data, used for all labels, inequality~\eqref{eq:local} holds for
every frame.
\end{proposition}

The proof is once more the computation of Section~\ref{sec:cert}, on a table
of $262{,}144$ lower bounds for $K^2$ for the kernel of the third window. A
frame is trivial when $\alpha\,(y_6-y_0)\ge x+0.038198$, that is, when its
spread is at least $14.927$. The search visits $572{,}959{,}453$ boxes and
accepts $286{,}479{,}727$ of them, with no failure.

\begin{proof}[Proof of Theorem~\ref{thm:third}]
By Propositions~\ref{prop:local3} and~\ref{prop:gram2} the pair
$(x,b)=(54525,\,37272)\cdot10^{-6}$ is a certificate for the critical line,
for the third window. Take $t=\frac{27}8$ in part~(c) of
Theorem~\ref{thm:further-general}. Then $4t-t^2=2.109375\ge2.10905=2+2x$ and
$e_0=\frac{999}{256}$, and with the bound $C''_0$ for $C(f)$ the quotient is
$\frac{50652159}{56956375}$.
\end{proof}

\begin{remark}[which certificate for which count]
For the average the level must be $2$: at the level $\frac{67}{20}$ a simple
pair off the line is credited with $2.18$ only, and it counts as much as a
double zero on the line, which costs $4$. For the zeros that are simple or on
the line the roles are different. The cheap zeros all count in full, so what
matters is the price $e_0$ per unit of the pairs off the line of multiplicity
at least two, and the level $t$ is chosen as large as the condition
$4t-t^2\ge2+2x_1$ allows. The objective of parts~(b) and~(c) is not that of
Theorem~\ref{thm:main} or of Theorem~\ref{thm:line}. The third window and its
local data were found for the objective of part~(c) at the level
$t=\frac{27}8$, which allows $x\le\frac7{128}$. The gain over the first
window is $0.00015$.
\end{remark}

\section{Structure and limits of the method}\label{sec:limits}

This section bounds from above what the method can give.
Table~\ref{tab:ceilings} collects the results for distinct zeros. The two
ceilings for the first window are formalised (Proposition~\ref{prop:six});
those for the two earlier windows (Proposition~\ref{prop:ceiling}) and the
last line (Corollary~\ref{cor:allwindows}) are certified by ball arithmetic.
The ceilings for the critical line are in
Proposition~\ref{prop:line-ceiling} and Corollary~\ref{cor:allwindows}.

\begin{table}
\caption{Proved lower bounds for $\liminf N_d/N$ and ceilings for the method.
Lower bounds are rounded down and ceilings up. The last line is for windows
with $V(f)\le10\,B(f)$ (Section~\ref{sec:allwindows}).}
\label{tab:ceilings}
\begin{tabular}{lccc}\toprule
Window & $(3-C(f))/2$ & Proved & Ceiling\\\midrule
flat, $f=1$ & $5/6$ & $5/6$ \cite{AF} & $5/6$\\
Montgomery--Taylor & $0.836250$ & $0.836250$ \cite{AF,Lam} & $0.83823$\\
window of \cite{K1} & $0.836228$ & $0.836992$ \cite{K1} & $0.83843$\\
first window & $0.833647$ & $0.839004$ & $0.83998$\\
first window, seven points & & $0.839004$ & $0.83916$\\
any window with $V(f)\le10\,B(f)$ & & & $0.84093$\\
\bottomrule\end{tabular}
\end{table}

\subsection{A ground-state problem}\label{sec:ground}
Consider any labelled configuration $(Y,D)$ on the line in place of the zeros,
with $N=n_1+2n_2$ units of multiplicity in an interval of length $N$, and let
\[
 \mathcal E(Y,D)=\sum_{i,j}d_id_jK(y_i-y_j)^2=n_1+4n_2+\sum_{i\ne j}d_id_jK(y_i-y_j)^2
\]
be its energy. The proof uses three facts about the zeros: $\tr A=N$,
$\tr A^2\le(C(f)+o(1))N$, and that the rescaled zeros lie in an interval of
length $N+o(N)$. An argument that uses nothing else applies to every such
configuration, so it cannot prove more than
\[
 Q^*(f)=\liminf_{N\to\infty}\ \min\Bigl\{\frac{n_1+n_2}{N}:\ \mathcal E(Y,D)\le C(f)\,N\Bigr\}.
\]
The minimum is over configurations on the line with labels in $\{1,2\}$
only, and with the energy constraint without the $o(1)$. Both restrictions
make the minimum larger, so $Q^*(f)$ is an upper bound for what can be proved,
not necessarily the exact limit. A certificate $(x_1,x_2,b)$ in the sense of
Section~\ref{sec:proof} is a linear lower bound for the pair energy: letting
$\varepsilon\to0$ for a fixed configuration and using $\varphi_c(t)\le t^2$,
it gives
$\mathcal E\ge(1+x_1)n_1+(4+x_2)n_2-b\,\spn(Y)-O(1)$, and with
$n_1=2(n_1+n_2)-N$, $n_2=N-(n_1+n_2)$ this is
$\mathcal E\ge(3+x_2-x_1-b)N-(2-2x_1+x_2)(n_1+n_2)-O(1)$. So for $\mu>0$ the
constant of Theorem~\ref{thm:general} satisfies
$(3-C+x_2-x_1-b)/(2-2x_1+x_2)\le Q^*(f)$; in particular $q\le Q^*(f)$ for the
first window. Conversely, since mixtures of configurations make the
set of attainable densities convex, one expects from convex duality that the
best certificate attains $Q^*(f)$; we do not need this and do not prove it.
Either way the problem is that of the least energy, as a function of the
densities of simple and double points, of a one-dimensional system with pair
interaction $d_id_jK(y_i-y_j)^2$.

\subsection{Trial states}
Upper bounds for $Q^*(f)$ come from periodic configurations and their
mixtures, in which long stretches of different periodic states follow each
other. The energy of a periodic state is a finite sum: $K^2$ is the Fourier
transform of $R=f*f$, which is supported in $[-1,1]$, so for period $P$ and
points $y_i$ in a period, by Poisson summation,
\[
 \sum_{i,i'}\sum_{n\in\mathbb Z}d_id_{i'}K(y_i-y_{i'}+nP)^2
 =\frac1P\sum_{|j|\le P}R\Bigl(\frac jP\Bigr)\Bigl|\sum_id_ie^{2\pi ijy_i/P}\Bigr|^2 .
\]

The simplest trial states explain the numbers. Let $e_{\mathrm{lat}}$ be the
pair energy per point of the unit lattice of simple points. For the first
window $K(n)=a_{|n|}/2$ for $1\le|n|\le6$ and $K(n)=0$ for $|n|\ge7$, so
$e_{\mathrm{lat}}=\frac12\sum a_j^2=0.02054$. Now take four consecutive
lattice sites and replace the four simple points by a simple point, a double
point, a simple point and an empty site, with the positions relaxed. This
saves pair energy, an amount $\Delta$ per cell, and costs $2$ in the diagonal
term. A mixture of the lattice and of such cells with a proportion $\rho_2$
of double points has energy $1+e_{\mathrm{lat}}+(2-\Delta)\rho_2$ per unit of
multiplicity, and the constraint $\mathcal E\le CN$ allows
\[
 \rho_2=\frac{C-1-e_{\mathrm{lat}}}{2-\Delta},\qquad\text{so}\qquad
 Q^*(f)\le1-\frac{C-1-e_{\mathrm{lat}}}{2-\Delta}.
\]
For $\Delta=e_{\mathrm{lat}}=0$ this is $(3-C)/2$. A window is good when its
lattice energy is large and merging saves little.

Table~\ref{tab:cell} gives the numbers for three windows.

\begin{table}
\caption{The cell ``simple, double, simple, empty'' as a trial state.}
\label{tab:cell}
\begin{tabular}{lcccc}\toprule
Window & $C-1$ & $e_{\mathrm{lat}}$ & $\Delta$ & Bound for $Q^*$\\\midrule
Montgomery--Taylor & $0.327499$ & $0.006127$ & $0.013471$ & $0.838225$\\
window of \cite{K1} & $0.327543$ & $0.006779$ & $0.014822$ & $0.838421$\\
first window & $0.332705$ & $0.020536$ & $0.044337$ & $0.840378$\\
\bottomrule\end{tabular}
\end{table}

The first window gives up $0.005$ in $C$ and gains $0.014$ in the
lattice energy (Figure~\ref{fig:windows}).

\begin{figure}
\begin{tikzpicture}
\begin{axis}[name=win, scale only axis, width=4.2cm, height=3.6cm, xmin=-0.5, xmax=0.5,
  axis line style={line width=0.5pt}, tick style={line width=0.5pt, color=black},
  xlabel={$u$}, ylabel={$f(u)$}, ymin=0.7, ymax=1.3]
\addplot[line width=1pt] table[x=u, y=fnew] {knausgard839win.dat};
\addplot[line width=1pt, dashed] table[x=u, y=fmt] {knausgard839win.dat};
\addplot[line width=1pt, dotted] table[x=u, y=fline] {knausgard839win.dat};
\end{axis}
\begin{axis}[at={(win.east)}, anchor=west, xshift=1.9cm, scale only axis, width=5.2cm, height=3.6cm,
  axis line style={line width=0.5pt}, tick style={line width=0.5pt, color=black},
  xmin=0, xmax=7.5, ymode=log, ymin=5e-7, ymax=1.5, xlabel={$t$}, ylabel={$K(t)^2$},
  unbounded coords=jump, xtick={0,1,...,7}]
\addplot[line width=0.7pt] table[x=t, y=knew] {knausgard839ker.dat};
\addplot[line width=0.7pt, dashed] table[x=t, y=kmt] {knausgard839ker.dat};
\addplot[line width=0.7pt, dotted] table[x=t, y=kline] {knausgard839ker.dat};
\addplot[only marks, mark=square*, mark size=1.5pt, mark options={fill=white, line width=0.5pt},
  each nth point=80, filter discard warning=false, skip coords between index={0}{1}]
  table[x=t, y=kmt] {knausgard839ker.dat};
\addplot[only marks, mark=*, mark size=1.6pt, mark options={fill=white, line width=0.5pt},
  each nth point=80, filter discard warning=false, skip coords between index={0}{1}]
  table[x=t, y=kline] {knausgard839ker.dat};
\addplot[only marks, mark=*, mark size=1.6pt,
  each nth point=80, filter discard warning=false, skip coords between index={0}{1}]
  table[x=t, y=knew] {knausgard839ker.dat};
\end{axis}
\end{tikzpicture}
\caption{Left: the first window (solid), the second window (dotted), and the Montgomery--Taylor window (dashed),
all normalised to $\int f=1$. Right: the squared kernels on a logarithmic
scale, with their values at the integers marked by filled circles, open
circles and squares, in the same order. At the integers $n=1,\dots,6$ the
kernel of the first window takes the values $a_n/2$, and it vanishes at the
larger integers, so the unit lattice carries the pair energy
$\frac12\sum a_j^2=0.0205$ per point, against $0.0087$ for the second window
and $0.0061$ for the Montgomery--Taylor window.}
\label{fig:windows}
\end{figure}
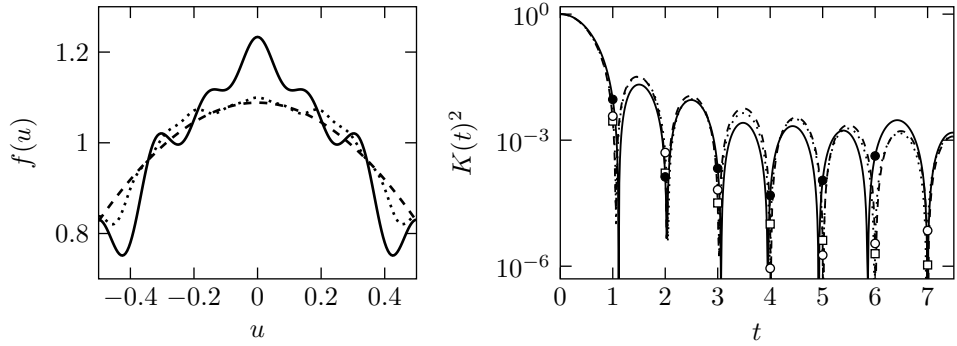

The entries of the table are enclosed by ball arithmetic, as in
Proposition~\ref{prop:ceiling} below. Better trial states exist for the first
window. A mixture of three states of period seven gives $0.839980$: six
simple points and one empty site; five simple points and one double point; and
the pattern simple, simple, double, empty, double, double, empty. They were
found by local optimisation over arrangements of simple points, double points
and empty sites on a few lattice sites, with the positions relaxed. We do not
claim that they are optimal, and for the two earlier windows we use only the
cell of the table.

\begin{proposition}\label{prop:ceiling}
For the first window, $Q^*(f)<0.83998$. For the window of~\cite{K1},
$Q^*<0.83843$, and for the Montgomery--Taylor window, $Q^*<0.83823$. The last
two are smaller than the constant $q$ of Theorem~\ref{thm:main}.
\end{proposition}

\begin{proof}
For each window we give two or three periodic states with rational positions,
all gaps at least $0.99$, and rational weights $w_k\ge0$. Their energies are
evaluated by the formula above in ball arithmetic, and $C(f)$ by rigorous
numerical integration, with the library FLINT/Arb~\cite{Arb}; the data and
the script are in the accompanying files. With $P_k$, $N_k$, $l_k$ and $E_k$ the
period, the multiplicity, the number of points and the energy of a period,
the weights satisfy $\sum_kw_kP_k\le\sum_kw_kN_k$ and
$\sum_kw_kE_k<C(f)\sum_kw_kN_k$. Let $M$ be a multiple of the denominators of
the $w_k$ and place $M_k=Mw_k$ periods of state $k$ one after the other, each
stretch directly after the previous one. The
configuration lies in an interval of length $\sum_kM_kP_k$. The energy of a
stretch is at most $M_kE_k$, because the energy of a period includes its
interactions with the whole periodic configuration and all terms are
nonnegative. Since $K(t)^2=O(t^{-2})$ and the gaps are bounded below, two
stretches interact with energy $O(\log M)$. So for large $M$ the
configuration is admissible in the definition of $Q^*$, and its proportion of
distinct points tends to $\sum_kw_kl_k/\sum_kw_kN_k$.
\end{proof}

For the flat window $f=1$ the kernel $K(t)=\sinc t$ vanishes at the nonzero
integers, so the unit lattice of simple points and the lattice of double
points with spacing two have no pair energy,
and their mixtures give $Q^*=\frac56$: for that window nothing can be added to
the plain bound $(3-C)/2$. For the first window $Q^*(f)$ lies between $q$ and
$0.83998$; with the first window the method cannot give $84\%$.

\subsection{Ceilings for certificates}\label{sec:six}
An inequality for seven consecutive points sees only pairs at index distance
at most six. The next statement bounds
every lower bound for the pair energy that is linear in $n_1$, $n_2$ and the
spread, whatever its proof. Its hypothesis for all pairs holds for every
certificate in the sense of Section~\ref{sec:proof}, as was shown in
Section~\ref{sec:ground}, and its hypothesis with $s=6$ holds for the
certificate of this paper, by Lemma~\ref{lem:transfer} with $\epsilon=0$.

\begin{proposition}[ceilings for certificates]\label{prop:six}
Let $f$ be the first window. Let $x_1,x_2$ and $b\ge0$ be real numbers
with $2-2x_1+x_2>0$ such that, for some constant $C_3$ and all
$\theta$-separated labelled configurations,
\[
 \sum_{0<|i-j|\le s}d_id_jK(y_i-y_j)^2\ \ge\ x_1n_1+x_2n_2-b\,\spn(Y)-C_3 .
\]
Then $(3-C(f)+x_2-x_1-b)/(2-2x_1+x_2)$ is less than $0.83916$ if $s=6$, less
than $0.83943$ if $s=8$, and less than $0.83957$ if $s=10$. If the sum is
taken over all pairs $i\ne j$ instead, it is less than $0.83998$.
\end{proposition}

\begin{proof}
Let $s=6$. We use three periodic states: the lattice of simple points with spacing
$1.02532$; six simple points and one empty site in a period of length
$7.07365$; and the pattern simple, simple, double, empty, double, double,
empty in a period of length $7.03513$. These lengths are rounded; the
rational positions are in the accompanying files. All
gaps exceed $\theta$. Applying the hypothesis to $M$ periods of state $k$,
dividing by $M$ and letting $M\to\infty$ gives $x_1n_{1,k}+x_2n_{2,k}-bP_k\le e_k$, where $e_k$
is the pair energy per period restricted to index distance at most six, a
finite sum of values of $K^2$. Take weights $w_k\ge0$, normalised so that
$\sum_kw_kN_k=1$, with $\sum_kw_kP_k\le1$ and
$\sum_kw_k(n_{1,k}+4n_{2,k}+e_k)\le C(f)$, and put $\bar q=\sum_kw_kl_k$. Then
$\sum_kw_kn_{1,k}=2\bar q-1$ and $\sum_kw_kn_{2,k}=1-\bar q$, so the weighted sum of
the three inequalities is
\[
 (2\bar q-1)x_1+(1-\bar q)x_2-b\ \le\ C(f)-(2\bar q-1)-4(1-\bar q),
\]
that is, $3-C(f)+x_2-x_1-b\le\bar q\,(2-2x_1+x_2)$. The weights in the accompanying
files give $\bar q=0.8391518\ldots$. For $s=8$ and $s=10$ the argument is the same,
each with three states, and gives $\bar q=0.8394222\ldots$ and
$\bar q=0.8395628\ldots$.

For all pairs we use the three states of period seven described before
Proposition~\ref{prop:ceiling}. The pair energy of a period is now a sum over
all periods. Since the period is an integer, $\sin^2(\pi t)$ takes the same
value at $t$ and at $t+7n$, and by~\eqref{eq:K} the sum over $n$ is a sum of
values of a rational function. Its terms with $n\le30$ are added exactly. For
$|t|\ge210$ one has $K(t)^2\le\sin^2(\pi t)\,E^2/(\pi^2t^2)$ with an explicit
$E$, and $1/t^2\le1/((t-\frac72)(t+\frac72))$ telescopes along $t+7n$, so the
rest is at most an explicit constant. This gives $\bar q=0.8399799\ldots$.

The four statements are formalised (Section~\ref{sec:lean}). The formal proof
uses the same states, with rational bounds for $\sin^2$ and for $\pi$ in
place of ball arithmetic.
\end{proof}

The certificate of this paper therefore realises at least $97\%$ of what
frames of seven points can add to $(3-C)/2$ for the first window, within the
scheme of Section~\ref{sec:matrix}. Going further needs longer frames: frames
of nine and of eleven points can give at most $0.83943$ and $0.83957$, and no
frame length can give $0.83998$.

\subsection{The critical line}
For Theorem~\ref{thm:line} the same two questions have the same kind of
answer. Only simple points interact, and a pair $(x,b)$ enters the bound
$(2-C(f)-b)/(1-x)$ through a lower bound for their pair energy. By
$\varphi_2(t)\le t^2$, a certificate $(x,b)$ for the critical line gives, for
a fixed set and $\varepsilon\to0$,
$\sum_{i\ne j}K(y_i-y_j)^2\ge x\,n-b\,\spn(Y)-C_2$, which is the hypothesis
of the next proposition for all pairs. A pair obtained from local data through
Proposition~\ref{prop:gram2} satisfies the hypothesis with $s=7$,
by~\eqref{eq:transfer8} with $\epsilon=0$. Unlike $Q^*(f)$, which bounds
every argument that uses only the trace, the energy and the length, the
proposition is a statement about the scheme of Section~\ref{sec:line} only:
its trial states place simple points and give the rest of the multiplicity the
energy $2$ per unit, with no interaction.

\begin{proposition}\label{prop:line-ceiling}
Let $f$ be the second window, of Section~\ref{sec:line}, and let $x<1$ and
$b\ge0$ be real numbers such that, for some constant $C_3$ and all
$\theta$-separated sets $Y$ of $n$ points,
\[
 \sum_{0<|i-j|\le s}K(y_i-y_j)^2\ \ge\ x\,n-b\,\spn(Y)-C_3 .
\]
If $s=7$, then $(2-C(f)-b)/(1-x)<0.67357$. If the sum is taken over all pairs
$i\ne j$ instead, then $(2-C(f)-b)/(1-x)<0.67383$.
\end{proposition}

\begin{proof}
Apply the hypothesis to $M$ periods of a periodic state with $n_k$ points and
period $P_k$. Let $e_k$ be its pair energy per period: the sum of
$K(y_i-y_j)^2$ over the ordered pairs $i\ne j$ with $y_i$ in one period,
restricted to $|i-j|\le7$ in the first case. Divide by $M$ and let
$M\to\infty$: $x\,n_k-b\,P_k\le e_k$. Take weights $w_k\ge0$ with
$\sum_kw_kP_k=1$ and $\sum_kw_k(e_k-n_k)+2-C(f)\le0$, and put
$\bar s=\sum_kw_kn_k$. The weighted sum is
$\bar s\,x-b\le\sum_kw_ke_k\le\bar s-(2-C(f))$, that is,
$2-C(f)-b\le\bar s\,(1-x)$. For $s=7$ we use two states with five points in
a period: one of length $7.04735$ with gaps $1.03627$, $1.03627$, $1.96582$,
$1.04317$, $1.96582$, and one of length $8.04125$ with gaps $1.04588$,
$1.97681$, $1.04588$, $1.98634$, $1.98634$. Ball arithmetic gives
$\bar s=0.6735660\ldots$; the
script and the data are in the accompanying files. For all
pairs the periods are $3$ and $7$, with gaps $1.03846$, $1.96154$ and
$1.03009$, $1.03009$, $1.95116$, $1.03749$, $1.95117$, and
$\bar s=0.6738216\ldots$; here the pair energy of a period is bounded as in
the proof of Proposition~\ref{prop:six}, with the first $60$ periods for the
state of period $3$. Both statements are formalised
(Section~\ref{sec:lean}).
\end{proof}

In these states gaps near $1$ and near $2$ alternate, so that all distances are
close to integers, where $K$ is small. The certificate of
Theorem~\ref{thm:line} realises at least $98\%$ of what frames of eight
points can add to $2-C(f)$ for the second window, within the scheme of
Section~\ref{sec:line}. Upper bounds of this kind were stated before. According
to~\cite{Trmdy}, a periodic configuration of $673$ points in a period of length
$1000$ with integer distances limits to $0.674826$ what a certificate can give
that uses only the pair energy of simple zeros, in the setting of that work.
Gebendorfer~\cite{GebCap} proved with a mixture of two periodic configurations
that for the windows $\cos(\sqrt2\,u)+\sum_{j=1}^{24}c_j\cos(2\pi ju)$,
normalised to integral one, and for certificates made from an inequality for
blocks of any length with a charge for each gap, the constant is below
$0.674034$. Proposition~\ref{prop:line-ceiling} is for one window of another
form, and its bounds are smaller.

\begin{remark}[the best certificate as an optimal-control problem]\label{rem:bellman}
In the terms of Remark~\ref{rem:control}, finding the best triple for frames
of seven points is dual to an average-cost problem for a chain with
interactions of range six: a triple can be admissible only if the least mean
value of $\text{energy}-x_1n_1-x_2n_2+b\cdot\text{length}$ per point is
nonnegative, and in the corresponding Bellman equation the storage function
$\Phi$ plays the role of the relative value function. Periodic trial states
are admissible trajectories and bound the value from the other side. We have
not proved that certificates of the form~\eqref{eq:local} attain the value of
this problem.
\end{remark}

\subsection{What limits the window}
Three requirements compete. The energy constant $C(f)$ should be small; its
minimum $1.3275$ is attained by the Montgomery--Taylor window~\cite{Mon75}. The lattice
energy should be large, which needs Fourier coefficients $a_j$ of some size,
and these increase $C$ through the term $\frac12\sum a_j^2$ in~\eqref{eq:C}.
Finally the edge value $f(\frac12)$ must be small enough for the majorant.
By Remark~\ref{rem:c} one needs $\Lambda<\frac c2\le1+\frac12\sqrt{2-2x_1}=1.6876$.
For a majorant of the form used in Section~\ref{sec:proof},
$f(\frac12)\le g(\frac12)=\gamma_0(1+\sinc\theta)^2$ forces
$\Lambda\ge2f(\frac12)/(\theta(1+\sinc\theta))=1.6700$ for the first window. So
here it is the edge value of $f$ that decides whether the separated blocks
escape the clipping.

\subsection{All windows}\label{sec:allwindows}
The bounds above are for one window at a time. A bound for all windows
follows from a construction of Alp\"oge and Furman~\cite[Section~7.2]{AF}. A
window enters the method only through the pair correlation of the zeros tested
against $R=f*f$, which is supported in $[-1,1]$. In~\cite{AF} a probability
law on periodic configurations is used to show that pair correlation in this
band cannot give more than about $0.682$ for simple zeros. The same
construction bounds the quantities of this paper. It has two parts, an
inequality and a law. The law of~\cite{AF} is not public, so we computed a law
of the same kind ourselves; it is given with the accompanying files, its
properties are certified, and its share of simple points is slightly smaller
than the one reported in~\cite{AF}.

\begin{definition}\label{def:law}
Let $N\ge2$ be an integer and $\tau\ge0$. A \emph{law of period $N$ with
error $\tau$} is a finitely supported probability law, with rational weights
$w_{\mathfrak c}$, on $N$-periodic configurations $\mathfrak c$ of distinct
points $\xi_i$ with labels $d_i\in\{1,2\}$ and total multiplicity $N$ per
period, such that the form factor
\[
 S(j)=\frac1N\sum_{\mathfrak c}w_{\mathfrak c}\Bigl|\sum_id_ie^{2\pi ij\xi_i/N}\Bigr|^2 ,
\]
in which the inner sum is over the points of one period, satisfies
$|N\,S(j)-j|\le\tau$ for $0<j<N$. Its \emph{share of simple points} is
$p=\frac1N\sum_{\mathfrak c}w_{\mathfrak c}\,n_1(\mathfrak c)$, where
$n_1(\mathfrak c)$ is the number of simple points of $\mathfrak c$ in a
period.
\end{definition}

The condition on $S(j)$ says that the law has the pair correlation of the
zeros in the band~\cite{Mon73}.

\begin{lemma}[{\cite[Section~7.2]{AF}, \cite{Claude26}}]\label{lem:stability}
Let $N\ge1$, $\tau\ge0$ and let $s_1,\dots,s_N$ be real numbers with
$|N^2s_j-j|\le\tau$ for $0<j<N$. Let $r$ be twice continuously differentiable
on $[0,1]$ with $r(1)=0$. Then
\[
 \Bigl|\sum_{j=1}^Ns_j\,r\Bigl(\frac jN\Bigr)-\int_0^1r(u)\,u\,du\Bigr|
 \ \le\ \Bigl(\frac1{6N^2}+\frac\tau{2N}\Bigr)\Bigl(|r'(1)|+\int_0^1|r''|\Bigr).
\]
\end{lemma}

It follows from two summations by parts; in the Lean library~\cite{Claude26}
it is the combination of \texttt{ceiling\_stability} and
\texttt{abs\_Efun\_le\_of\_nearCUE} (directory \texttt{PairCeiling}), applied
to a continuously differentiable extension of $r$ from $[0,1]$: hypotheses and
conclusion involve $r$ on $[0,1]$ only.

\begin{theorem}[a ceiling for all windows]\label{thm:allwindows}
Suppose that there is a law of period $N$ with error $\tau$ and share of
simple points $p$, and put $\tau_0=\frac1{6N^2}+\frac\tau{2N}$. Let $f$ be an
even probability density on $I$, twice continuously differentiable on $I$,
let $R=f*f$,
\begin{align*}
 B(f)&=\iint|s-t|\,f(s)f(t)\,ds\,dt=2\int_0^1uR(u)\,du,\\
 V(f)&=f(\tfrac12)^2+\int_0^1|R''(u)|\,du ,
\end{align*}
and $\delta_0=2\tau_0V(f)/(B(f)+2\tau_0V(f))$.
\begin{enumerate}
\item[(a)] $Q^*(f)\le\frac{1+p}2+\frac{1-p}2\,\delta_0$. In particular no
certificate for $f$ with $2-2x_1+x_2>0$ gives more.
\item[(b)] Suppose that all gaps in the configurations of the law are at least
$g$, and let $0<\theta\le\min(g,1)$. Let $K$ be the kernel of $f$, and let
$x<1$ and $b\ge0$ be such that, for some constant $C_3$ and all
$\theta$-separated sets $Y$ of $n$ points,
$\sum_{i\ne j}K(y_i-y_j)^2\ge x\,n-b\,\spn(Y)-C_3$. Then
$(2-C(f)-b)/(1-x)\le p+(1-p)\,\delta_0$.
\end{enumerate}
\end{theorem}

Part~(b) is the statement for the scheme of Section~\ref{sec:line}: its
hypothesis is that of Proposition~\ref{prop:line-ceiling} for all pairs, now
for any window and any separation up to $\min(g,1)$.

\begin{proof}
(a) A configuration with labels in $\{1,2\}$ and $N$ units of multiplicity per
period, of which $n_1$ are simple points, has $(N+n_1)/2$ points per period.
So the law has on average $\frac{1+p}2$ points per unit of multiplicity. By
the Poisson formula, the energy per period of a configuration $\mathfrak c$ of
the law is $E_{\mathfrak c}=N\,R(0)+2\sum_{j=1}^{N-1}R(j/N)\,S_{\mathfrak c}(j)$
with $S_{\mathfrak c}(j)=|\sum_id_ie^{2\pi ij\xi_i/N}|^2/N$. Hence the mean
energy of the law per unit of multiplicity is
\[
 R(0)+2\sum_{j=1}^{N}\frac{S(j)}{N}\,R\Bigl(\frac jN\Bigr)
 \ \le\ R(0)+2\int_0^1uR(u)\,du+2\tau_0V(f)=C(f)+2\tau_0V(f),
\]
by Lemma~\ref{lem:stability} with $s_j=S(j)/N$ and $r=R$: on $[0,1]$ the
function $R$ is twice continuously differentiable up to the endpoints,
$R(1)=0$ and $R'(1)=-f(\frac12)^2$. The unit lattice of simple points has
energy $R(0)=C(f)-B(f)$ per point. A mixture, in the sense of
Proposition~\ref{prop:ceiling}, of the lattice with a share $\delta$ of the
multiplicity and of the configurations of the law with the share $1-\delta$
has energy at most $C(f)-\delta B(f)+(1-\delta)2\tau_0V(f)$ per unit of
multiplicity, and $\delta+(1-\delta)\frac{1+p}2$ points. All periods equal the
multiplicities, so the mixture is admissible in the definition of $Q^*$ for
every rational $\delta>\delta_0$. The argument of
Proposition~\ref{prop:ceiling} uses only that the weights are rational, that
there are finitely many states and that their points are distinct. Let
$\delta$ tend to $\delta_0$. The last sentence follows from
Section~\ref{sec:ground}.

(b) Let $Y_{\mathfrak c}$ be the set of simple points of a configuration
$\mathfrak c$ of the law. It has period $N$ and $n_1(\mathfrak c)$ points per
period, and it is $\theta$-separated. Let $e_{\mathfrak c}$ be its pair energy
per period, the sum of $K(y-y')^2$ over the ordered pairs of distinct points
with $y$ in one period. All terms of the energy of $\mathfrak c$ are
nonnegative and its diagonal terms add up to $n_1+4n_2$ per period, so
$e_{\mathfrak c}\le E_{\mathfrak c}-n_1(\mathfrak c)-4n_2(\mathfrak c)$. As in
the proof of Proposition~\ref{prop:line-ceiling}, the hypothesis applied to
$M$ periods, divided by $M$, gives for $M\to\infty$
\[
 x\,n_1(\mathfrak c)-b\,N\le e_{\mathfrak c},
 \qquad\text{and}\qquad x-b\le R(0)-1=C(f)-B(f)-1
\]
for the unit lattice, which is $\theta$-separated as well. Multiply the first
inequality by $(1-\delta_0)w_{\mathfrak c}/N$, sum over $\mathfrak c$, and add
$\delta_0$ times the second. With $\bar s=\delta_0+(1-\delta_0)p$ the left side
is $\bar s\,x-b$. The mean of $n_1+4n_2$ per unit of multiplicity is
$p+2(1-p)$, so by the bound for the mean energy in~(a) the right side is at
most
\[
 (1-\delta_0)\bigl(C(f)+2\tau_0V(f)-2+p\bigr)+\delta_0\bigl(C(f)-B(f)-1\bigr)
 =\bar s-(2-C(f)),
\]
where the terms $(1-\delta_0)2\tau_0V(f)$ and $\delta_0B(f)$ cancel by the
choice of $\delta_0$. Hence $2-C(f)-b\le\bar s\,(1-x)$, and
$\bar s=p+(1-p)\,\delta_0$.
\end{proof}

\begin{proposition}[a law of period $256$]\label{prop:law}
There is a law of period $256$ with error $10^{-30}$, supported on $256$
configurations with rational points, in which all gaps are at least $0.998$,
and with share of simple points $p=0.6818276195\ldots$
\end{proposition}

\begin{proof}
The law is in the accompanying files: the points have twelve decimals, and the
weights have the denominator $10^{60}$. That the labels are $1$ or $2$, that
the multiplicities add up to $256$ in each configuration, the gaps, that the
weights are nonnegative with sum $1$, and the value of $p$ are checked in
rational arithmetic. Each $256\,S(j)$ is a polynomial with rational
coefficients in the numbers $\cos(2\pi j\xi_i/256)$ and
$\sin(2\pi j\xi_i/256)$ with rational $\xi_i$. It is enclosed by ball
arithmetic with FLINT/Arb~\cite{Arb} at $600$ bits, and the enclosure lies
within $10^{-30}$ of $j$. A second program, which shares no code with the
first and uses the interval arithmetic of mpmath~\cite{mpmath}, repeats all
checks.
\end{proof}

\begin{corollary}\label{cor:allwindows}
Let $f$ be as in Theorem~\ref{thm:allwindows} with $V(f)\le10\,B(f)$. Then
$Q^*(f)<0.84093$. If $0<\theta\le0.998$ and $(x,b)$ is as in
Theorem~\ref{thm:allwindows}\textup{(b)}, then $(2-C(f)-b)/(1-x)<0.68185$.
\end{corollary}

\begin{proof}
For the law of Proposition~\ref{prop:law},
$\tau_0=\frac1{6\cdot256^2}+\frac{10^{-30}}{512}<2.5431316\cdot10^{-6}$, and
$\delta_0$ increases with $V(f)/B(f)$, so
$\delta_0\le20\tau_0/(1+20\tau_0)<5.0861\cdot10^{-5}$. With
$p<0.6818277$ the two bounds of the theorem are below $0.84093$ and $0.68185$.
\end{proof}

For the four windows of Table~\ref{tab:ceilings}, $V(f)/B(f)$ is at most
$3$, $5.2$, $5.3$ and $9.0$, and for the second window it is at most $5.7$
(ball arithmetic), so the corollary applies to them. We do not know whether
laws with triple points give a smaller constant, nor whether some window
attains $(1+p)/2$.

\emph{How the law was found.} For a finite set of configurations, the laws
with $N\,S(j)=j$ for $0<j<N$ that are supported on this set are the solutions
of $N$ linear equations in nonnegative weights, and $p$ is linear in the
weights. The least $p$ is therefore the value of a linear programme. A further
configuration $\mathfrak c$ improves the programme if its reduced cost
\[
 \frac{n_1(\mathfrak c)}N-y_0-\sum_{j=1}^{N-1}y_j\Bigl|\sum_id_ie^{2\pi ij\xi_i/N}\Bigr|^2
\]
is negative, where $y_0,\dots,y_{N-1}$ are the dual variables of the current
solution. To find such a configuration is again a ground-state problem, now
for labelled points with the pair interaction
$-d_id_{i'}\sum_jy_j\cos(2\pi j(\xi_i-\xi_{i'})/N)$. We alternate between the
linear programme, solved with HiGHS~\cite{HiGHS}, and a local search for
configurations of negative reduced cost; this is column
generation~\cite{GG61}. The search starts from perturbations
of the configurations in use and from neighbours of them in which two simple
points are merged, a double point is split, or a double point is moved. It
runs in floating point and is a heuristic. This is harmless, because
Theorem~\ref{thm:allwindows} holds for every law: at the end the points are
rounded to rationals, the weights are recomputed for the rounded points from
the $N$ linear equations in high precision and rounded to rationals with
sum~$1$, and the result is certified as in the proof of
Proposition~\ref{prop:law}. After the rounding the equations hold up to the
error, which is why the definition allows one. Each period starts from
configurations of the previous one placed one after the other. For the period
$256$ a second stage follows. From each configuration in use it runs a chain
of moves that change the labels: a double point is moved, two simple points
are merged, or a double point is split. The positions are relaxed after every
move, and a move that raises the reduced cost is accepted with the Metropolis
probability. In $700$ rounds this lowers $p$ by about $1.4\cdot10^{-6}$.
Table~\ref{tab:laws} gives the laws obtained in this way; each of them is
certified, and the last one is the law of Proposition~\ref{prop:law}. The
programs, the laws and the logs are in the accompanying files.

\begin{table}
\caption{Certified laws of period $N$ with error $10^{-30}$: the share of
simple points $p$, and the bounds of Theorem~\ref{thm:allwindows} for windows
with $V(f)\le10\,B(f)$. All values are rounded up. The last row is the law
after the second stage of the search.}
\label{tab:laws}
\begin{tabular}{rcccc}\toprule
$N$ & $p$ & $(1+p)/2$ & Bound (a) & Bound (b)\\\midrule
$16$ & $0.6828368$ & $0.8414184$ & $0.84346$ & $0.68692$\\
$32$ & $0.6820902$ & $0.8410451$ & $0.84157$ & $0.68313$\\
$64$ & $0.6818908$ & $0.8409454$ & $0.84108$ & $0.68215$\\
$128$ & $0.6818404$ & $0.8409202$ & $0.84096$ & $0.68191$\\
$256$ & $0.6818290$ & $0.8409145$ & $0.84093$ & $0.68185$\\
$256$ & $0.6818277$ & $0.8409139$ & $0.84093$ & $0.68185$\\
\bottomrule\end{tabular}
\end{table}

\emph{Comparison with~\cite{AF}.} In~\cite[Section~7.2]{AF} and in the Lean
library~\cite{Claude26} a law of period $256$ with labels in $\{1,2\}$ is
used. The library records its error $3\cdot10^{-40}$, the value
$p_0=0.68182868746\ldots$ of its share of simple points, and interval
enclosures of its form factor, computed outside Lean. According to the library
the law is the optimal solution of a linear programme over periodic
configurations and is available from the authors of~\cite{AF}; it is not among
the public files, and we did not use it. That the points of each configuration
are distinct is not stated there. If they are,
Theorem~\ref{thm:allwindows}\textup{(a)} with these values gives
$Q^*(f)<0.84093$ for $V(f)\le10\,B(f)$. Our laws were computed from this
description, without their data. After the first stage of the search the
share of simple points exceeds $p_0$ by less than $3.1\cdot10^{-7}$. After the
second stage it is smaller than $p_0$ by more than $1.06\cdot10^{-6}$, and the
error of that law, which the second check bounds as well, is below
$10^{-50}$. So $p_0$, the optimum of the finite programme of~\cite{AF}, is
not the least share of simple points of a law of period $256$; we do not know
the least share. The bound $0.84093$ of Corollary~\ref{cor:allwindows} does
not change, and it does not depend on the law of~\cite{AF}. Our data and
programs are public.

In summary, $q\le Q^*(f)<0.83998$ for the first window, and $Q^*(f)<0.84093$
for every window with $V\le10B$. Less than $0.002$ separates
Theorem~\ref{thm:main} from this limit.

\section{Verification and nature of this work}\label{sec:lean}

\emph{The formal statement.} The accompanying files contain a
Lean~4~\cite{Lean4} development, built with Mathlib~\cite{mathlib}, that proves
\begin{verbatim}
theorem distinct_zeros_839 :
    ∀ ε > 0, ∃ T₀ : ℝ, ∀ T ≥ T₀,
      (1645064 / 1960733 - ε) * (Ncount 0 T : ℝ) ≤ Ndist 0 T
theorem simple_zeros_678 :
    ∀ ε > 0, ∃ T₀ : ℝ, ∀ T ≥ T₀,
      (1329395 / 1960733 - ε) * (Ncount 0 T : ℝ) ≤ Nsimple 0 T
theorem simple_critical_zeros_673 :
    ∀ ε > 0, ∃ T₀ : ℝ, ∀ T ≥ T₀,
      (1669159 / 2478195 - ε) * (Ncount 0 T : ℝ) ≤ N0simple 0 T
theorem average_simple_and_critical_8367 :
    ∀ ε > 0, ∃ T₀ : ℝ, ∀ T ≥ T₀,
      (2073677 / 2478195 - ε) * (Ncount 0 T : ℝ)
        ≤ ((Nsimple 0 T : ℝ) + (N0mult 0 T : ℝ)) / 2
theorem simple_or_critical_8891 :
    ∀ ε > 0, ∃ T₀ : ℝ, ∀ T ≥ T₀,
      (3376433 / 3797325 - ε) * (Ncount 0 T : ℝ) ≤ NsimpleOrLine 0 T
theorem simple_or_critical_8893 :
    ∀ ε > 0, ∃ T₀ : ℝ, ∀ T ≥ T₀,
      (50652159 / 56956375 - ε) * (Ncount 0 T : ℝ)
        ≤ NsimpleOrLine 0 T
\end{verbatim}
The counting functions count zeros of Mathlib's \texttt{riemannZeta}
with real part in $(0,1)$ and imaginary part in $(0,T]$. \texttt{Ncount}
counts them with multiplicity and \texttt{Ndist} without. \texttt{Nsimple}
counts those of multiplicity one, and \texttt{N0simple} those of multiplicity
one with real part $\frac12$. \texttt{N0mult} counts, with multiplicity,
those with real part $\frac12$, and \texttt{NsimpleOrLine} counts, with
multiplicity, those that have multiplicity one or real part $\frac12$. The
multiplicity is the order of vanishing. The first five are definitions
of~\cite{Claude26}, so the first three statements have the form of the formal
theorems there for $\frac56$, $0.83625$ and $0.67250$. Four files import
only Mathlib and contain the definitions and the six statements, with the
proofs left open by design: \texttt{Challenge839.lean},
\texttt{Challenge673.lean}, \texttt{ChallengeFurther.lean} and
\texttt{Challenge889.lean}. The four files with \texttt{Solution} in place
of \texttt{Challenge} in their names prove the statements. Table~\ref{tab:status} says
how each result of this paper is established, and Table~\ref{tab:map} where
the formal results are in the files.

\emph{Compared with the first paper.} In~\cite{K1} the matrix inequalities,
the count of the zeros and the arithmetic were formalised, and the rest
entered the formal statements as hypotheses: the operator and its reduction to
matrices, the energy asymptotic, the smoothing, the passage from seven points
to all points, and the seven-point inequality itself, which was imported
from~\cite{Trmdy}. Here each of these steps is proved in Lean, so the formal
statements are about the zeros of the zeta function and have no hypotheses.
Since $q$ is larger than the constant of~\cite{K1}, the first statement
contains a formal proof of the bound of that paper as well.

\begin{table}
\caption{How the results are established. ``Lean'' means a formal proof from
the axioms of Mathlib, with the further assumption given in the last column.}
\label{tab:status}
\begin{tabular}{>{\raggedright\arraybackslash}p{0.40\textwidth}>{\raggedright\arraybackslash}p{0.22\textwidth}>{\raggedright\arraybackslash}p{0.27\textwidth}}\toprule
Result & Established by & Further assumption\\\midrule
Theorem~\ref{thm:main}, Corollary~\ref{cor:simple},
Proposition~\ref{prop:local} & Lean & the axiom that the first search
returns \texttt{true}\\\addlinespace
Theorem~\ref{thm:line}, Proposition~\ref{prop:local2} & Lean & the axiom
that the second search returns \texttt{true}\\\addlinespace
Theorem~\ref{thm:further} & Lean & the second of these axioms for the first
bound, the first for the second bound\\\addlinespace
Theorem~\ref{thm:third}, Proposition~\ref{prop:local3} & Lean & the axiom
that the third search returns \texttt{true}\\\addlinespace
Propositions~\ref{prop:six} and~\ref{prop:line-ceiling} & Lean & none\\\addlinespace
Sections~\ref{sec:matrix}, \ref{sec:proof} and~\ref{sec:line} with general
parameters & written proofs & none; Lean has the instances that are
used\\\addlinespace
Proposition~\ref{prop:ceiling}, Table~\ref{tab:cell} & ball arithmetic,
outside Lean & the library FLINT/Arb\\\addlinespace
Theorem~\ref{thm:allwindows}, Proposition~\ref{prop:law},
Corollary~\ref{cor:allwindows} & written proof and ball arithmetic, outside
Lean & the library FLINT/Arb; a second program with mpmath repeats the
check\\\bottomrule
\end{tabular}
\end{table}

\begin{table}
\caption{Where the formal results are in the accompanying Lean files. The
statement files are named \texttt{Challenge} in place of \texttt{Solution}.
The first, second and third search are the evaluations by
\texttt{native\_decide} in the files \texttt{Local/RunAll.lean} of the
directories \texttt{Distinct839}, \texttt{Simple673} and \texttt{Line889}.}
\label{tab:map}
\begin{tabular}{@{}>{\raggedright\arraybackslash}p{0.22\textwidth}>{\raggedright\arraybackslash}p{0.46\textwidth}>{\raggedright\arraybackslash}p{0.10\textwidth}>{\raggedright\arraybackslash}p{0.10\textwidth}@{}}\toprule
Result & Lean statement and file & Window & Search\\\midrule
Theorem~\ref{thm:main} & \texttt{distinct\_zeros\_839}, \texttt{Solution839.lean} & first & first\\\addlinespace
Corollary~\ref{cor:simple} & \texttt{simple\_zeros\_678}, \texttt{Solution839.lean} & first & first\\\addlinespace
Theorem~\ref{thm:line} & \texttt{simple\_critical\_zeros\_673}, \texttt{Solution673.lean} & second & second\\\addlinespace
Theorem~\ref{thm:further}, first bound & \texttt{average\_simple\_and\_}\allowbreak\texttt{critical\_8367}, \texttt{SolutionFurther.lean} & second & second\\\addlinespace
Theorem~\ref{thm:further}, second bound & \texttt{simple\_or\_critical\_8891}, \texttt{SolutionFurther.lean} & first & first\\\addlinespace
Theorem~\ref{thm:third} & \texttt{simple\_or\_critical\_8893}, \texttt{Solution889.lean} & third & third\\\addlinespace
Proposition~\ref{prop:six} & \texttt{ceiling\_range6}, \texttt{ceiling\_range8}, \texttt{ceiling\_range10}, \texttt{ceiling\_all}, \texttt{SolutionCeilings.lean} & first & none\\\addlinespace
Proposition~\ref{prop:line-ceiling} & \texttt{ceiling\_line\_range7}, \texttt{ceiling\_line\_all}, \texttt{SolutionCeilings.lean} & second & none\\\bottomrule
\end{tabular}
\end{table}

\emph{What is reused.} The development extends the library
of~\cite{Claude26,AF}, which derives from Mathlib the explicit formula, the
Riemann--von Mangoldt formula~\cite{Tit86} and the Montgomery--Vaughan
inequality~\cite{MV74}. It has
the machinery for the trace and the second moment of a discretised form of
the zero operator for a class of windows, and the resulting statements for
the zeta function for two windows. That operator is a finite matrix indexed by a
grid of heights; zeros enter through the values of the transform of the window
at the grid, and the bandwidth is $\lambda\log T$ with $\lambda<1$. Two
constants in that library were relaxed: the lower bounds $\frac12$ for two
normalised moments of the window became $\frac18$ and $\frac14$. The reason
is that our window is not monotone on $[0,\frac12]$ and has to be scaled down to satisfy
the variation bounds of the class. The patch changes $14$ files and no
headline statement of the library. Mathlib itself is used unchanged, at the
revision fixed by the library.

\emph{What is new.} In order of the proof:
\begin{itemize}
\item the window $f$ in the class of the library, its moments, and the energy
constant $\frac A\lambda+\lambda B$ with $A=\int f^2$,
$B=2\int_0^1uR(u)\,du$ and $A+B<C_0$, so that a bandwidth $\lambda<1$ close
to $1$ gives the energy bound with $C_0$;
\item the second moment of the zero operator of the zeta function for a
general window of the class, which the library had for the flat and the
Montgomery--Taylor window;
\item the Gram matrix of zeros on the critical line: for zeros away from the
ends of the range its entries are within $o(1)$ of $K(y_i-y_j)$, and on
separated sets it is bounded by $\frac c2$ (Lemma~\ref{lem:maj}, with the
majorant, its Fourier transform and $f\le g$ proved in Lean);
\item Theorem~\ref{thm:gram}, with $\eta_0=\frac1{100}$ and $C_1=152$, and
Lemma~\ref{lem:threshold} in eigenvalue form;
\item Lemma~\ref{lem:transfer} with the checks of the frozen tables;
\item Proposition~\ref{prop:local}, see below;
\item the counting argument of Section~\ref{sec:proof}.
\end{itemize}
The formal proof follows the written one, with these differences, none of
which changes the constant: it works with the discretised operator
of~\cite{Claude26} at bandwidth $\lambda<1$ on the ranges $(T,2T]$ and then
sums over dyadic ranges; zeros within $\sqrt T$ of the ends of a range are
moved into the indefinite part, together with the zeros of multiplicity at
least three and the pairs off the line; the pair blocks are estimated with an
explicit loss per point; the Gram bound is stated for a positive semidefinite
matrix with the given entries and for distinct points, which is the case in
the application; $f\le g$ is checked on $42$ dyadic intervals; and
$\beta$ is obtained from thirteen cells of the kernel table and $|K'|\le\pi$.

\emph{The critical line.} Theorem~\ref{thm:line} has its own development,
which repeats this chain for the second window at the level $c=2$ and reuses
every part of the first that does not depend on the window or on the labels.
New in it are Proposition~\ref{prop:gram2}, with $\eta_0=\frac1{100}$ and
$C_1=44$; the inequality~\eqref{eq:transfer8} for frames of eight points,
proved in the form with pairs at index distance at most seven; and the count of
Section~\ref{sec:line}. Three points differ from the written proof and do not
change the constant. Simple zeros within $\sqrt T$ of the ends of a range go
into the remainder, where each costs $4$ and carries multiplicity $1$; there
are $o(N)$ of them. The Gram entries are known up to a tolerance, so the
diagonal entries need not equal $1$, and a pair block is estimated by
comparison and not by its eigenvalues $1\pm u$. And the formal proof uses
$\beta=\frac1{20}$, which is all that $x\le\beta$ needs. The bound for
separated points is used with the constant $\frac{19}{10}<2$, and $f\le g$ is
checked on $14$ dyadic intervals.

\emph{The third window.} Theorem~\ref{thm:third} has a development of the
same form, which repeats the chain of the critical line for the third window
and ends with the count of part~(c) of Theorem~\ref{thm:further-general}.
The three points above apply again, with the cost $t^2$ in place of $4$ for a
simple zero near the ends of a range. The formal proof uses
$\beta=\frac3{50}$ and the constant $\frac{17}{10}$, and $f\le g$ is checked
on $28$ dyadic intervals.

\emph{The computation.} The kernel table of Section~\ref{sec:cert} is defined
in Lean in exact rational arithmetic, and its soundness, the statement that
$K(t)^2\ge v_k/2^{28}$ on each cell, is a theorem. The branch-and-bound
search is a Lean function on integers, and its soundness is a theorem: if it
returns \texttt{true} on a sound table, Proposition~\ref{prop:local} holds.
What Lean's kernel does not check is one evaluation, the statement that the
search returns \texttt{true} on the table. It is carried out by the tactic
\texttt{native\_decide}: the table is evaluated by Lean's interpreter and the
search by compiled code, and the outcome is recorded as an axiom. This trusts
the Lean compiler, interpreter and runtime for this one computation. The
kernel can evaluate the search on single small boxes, as in the third check
below, but the whole search is out of reach for it in practice. Certificates
of other forms have been checked by the kernel itself: one with $5{,}651$
cases in~\cite{Lavery}, and one through about $1{,}500$ generated files
in~\cite{DSouza}. Three
checks support the evaluation. The table computed in Lean, written to a file, agrees with
the table of the original certificate in all $209{,}010$ entries. A version of
the search with counters, built from the same function for a single box,
visits $60{,}467{,}309$ boxes on it and accepts $30{,}233{,}655$, the numbers
of the reference programs in C++ and Python in the accompanying files, which
run the same search. And on four small boxes of the domain, among them one
around the point of smallest slack, the first test of the search is evaluated
by the kernel itself, which proves the inequality there with the standard
axioms.

Proposition~\ref{prop:local2} is proved by a second run of the same kind,
for boxes of seven gaps, with its own table and its own axiom. This search is
about nine times as long as the first, and it is evaluated in parts: the boxes
that are left after $32$ bisections are dealt out to $96$ lists, which the
compiled code examines as separate tasks, and a theorem states that the
inequality holds if the evaluation in parts returns \texttt{true}. The
statement that is proved does not depend on the division. The table computed
in Lean agrees with the reference table in all $262{,}144$ entries, and the
version with counters visits $535{,}332{,}163$ boxes and accepts $267{,}666{,}082$,
the numbers of the C++ program. Four small boxes are again checked by the kernel
itself, among them one around the point where a numerical search found the
smallest slack of the inequality, about $1.3\cdot10^{-5}$.

Proposition~\ref{prop:local3} is proved by a third run, with its own table
and its own axiom. The table computed in Lean agrees with the reference table
in all $262{,}144$ entries, and the version with counters visits
$572{,}959{,}453$ boxes and accepts $286{,}479{,}727$, again the numbers of
the C++ program. Four small boxes are checked by the kernel itself here too,
two of them around the two points where a numerical search found the smallest
slack, about $1.0\cdot10^{-5}$ and $2.0\cdot10^{-5}$.

\emph{Axioms and size.} For each of the six theorems the command
\texttt{\#print axioms} lists the three standard axioms \texttt{propext},
\texttt{Classical.choice} and \texttt{Quot.sound}, and one axiom of a
computation. That axiom records that the evaluation of
\texttt{checkAll kTabArray}, or of \texttt{checkAllPar kTabArray 32 96} for the
second run, returned \texttt{true}, for the table in question:
Theorem~\ref{thm:main}, Corollary~\ref{cor:simple} and the second bound of
Theorem~\ref{thm:further} depend on the first run, Theorem~\ref{thm:line}
and the first bound of Theorem~\ref{thm:further} on the second, and
Theorem~\ref{thm:third} on the third. Everything else is checked by the
kernel; in particular the implications from
Propositions~\ref{prop:local}, \ref{prop:local2} and~\ref{prop:local3} to the theorems and the
soundness of the search depend on the three standard axioms only. The proof
developments contain no \texttt{sorry}; it occurs only in the five statement
files \texttt{Challenge*.lean}, where the proofs are left open by design. The
development for
Theorem~\ref{thm:main} has about $13{,}300$ lines in $72$ files and the one
for Theorem~\ref{thm:line} about $10{,}000$ lines in $57$ files; the library
they build on has about $102{,}000$. Theorem~\ref{thm:further} adds about
$1{,}900$ lines in $12$ files: Lemma~\ref{lem:threshold2}, the two counts of
Section~\ref{sec:further}, and the instances. The development for
Theorem~\ref{thm:third} has about $9{,}100$ lines in $54$ files. The statements of Section~\ref{sec:matrix} are formalised for the
constants of Section~\ref{sec:proof}, not for general parameters.

\emph{The ceilings.} Propositions~\ref{prop:six}
and~\ref{prop:line-ceiling} are formalised in a separate development, of about
$3{,}400$ lines in $23$ files. Its six theorems depend on the three standard
axioms only; no compiled computation enters. The file
\texttt{ChallengeCeilings.lean} states them from Mathlib alone, with the two
kernels written out and the two energy constants defined by the integrals
in~\eqref{eq:energy-intro}, and
\texttt{SolutionCeilings.lean} proves these statements. The hypothesis of
Proposition~\ref{prop:six} is the definition \texttt{PairCert}, with the
arguments $s$, $x_1$, $x_2$ and $b$:
there is a constant $C_3$ such that the displayed inequality holds for all
$\theta$-separated labelled configurations and every $R\ge0$ that bounds
their spread, with $b\,R$ in place of $b\,\spn(Y)$. The conclusion is stated
with $C(f)$ as defined there. Two identities connect it with the proofs of
the theorems, and both are formalised: $C(f)=A+B$ with $A=\int f^2$ and
$B=2\int_0^1uR(u)\,du$, which is the change of variables $u=s-t$ and is proved
for every continuous function, and $A+B$ is the value at $\lambda=1$ of the
energy constant used in the proofs of Theorems~\ref{thm:main}
and~\ref{thm:line}. The trial states have rational positions, the values
of $K^2$ are bounded above through Taylor polynomials in rational arithmetic,
and the sums over the periods are treated as in the proof of
Proposition~\ref{prop:six}. Two further statements connect the definition
with the rest of the paper. An inequality of the form of
Theorem~\ref{thm:gram}, with any constants in place of ours and with exact
kernel values, implies the hypothesis for all pairs; this depends on the
standard axioms only. And the certificates of Theorems~\ref{thm:main}
and~\ref{thm:line} satisfy the hypothesis with $s=6$ and with $s=7$; these
are two statements, one for each theorem, and each depends on the axiom of
the respective search. The ceilings are formalised for the separation
$\frac{17}{20}$; the statement of Remark~\ref{rem:c} about other separations
is a written argument.

Not formalised are the quantity $Q^*(f)$ of Section~\ref{sec:ground} and the
passage from periodic states to it, the bounds of
Proposition~\ref{prop:ceiling} for the two earlier windows,
Table~\ref{tab:cell}, and Section~\ref{sec:allwindows}, of which only
Lemma~\ref{lem:stability} is in the Lean library. Their numerical parts are
certified by ball arithmetic with FLINT/Arb~\cite{Arb}, outside Lean. For the
law of Proposition~\ref{prop:law} the data are exact rational numbers, and a
second program with a different interval library repeats the check. The
programs that searched for the law are in Python and, for the local
optimisation and for the chains of the second stage, in C++; they are in the
accompanying files with their tests,
and nothing in the proof depends on their correctness. No statement of this
paper rests on unvalidated numerical computation: outside Lean, the
computations that enter a proof use exact rational arithmetic or interval
arithmetic with rigorous error bounds. Floating point without error control
was used only to find the certificates, the trial states and the law.

\emph{Nature of this work.} This paper builds on the first paper in this
series~\cite{K1}. Like that paper, it is primarily an experiment in AI-assisted
mathematical research. The arguments, the
computations, the Lean formalisation, the figures and much of the text were
developed with substantial assistance from AI systems (OpenAI Codex and
Anthropic Claude) under the author's direction. The certificate of
Proposition~\ref{prop:local}, the separation argument and the telescoping
corrections were developed with OpenAI Codex. The formalisation, the analysis
of Section~\ref{sec:limits}, the window and the certificate of
Theorem~\ref{thm:line}, the figures and much of the text were developed with
Anthropic Claude, and drafts of the paper were reviewed by the same systems in
the role of referees. The author defined and steered the research programme
and takes full responsibility for the content.

\appendix
\section{Certificate data}\label{app:data}

Tables~\ref{tab:weights}, \ref{tab:rewards} and~\ref{tab:knots} give the
local data of Proposition~\ref{prop:local}, Tables~\ref{tab:line}
and~\ref{tab:lineknots} those of Proposition~\ref{prop:local2}, and
Tables~\ref{tab:third} and~\ref{tab:thirdknots} those of
Proposition~\ref{prop:local3}. All integers
in the tables are to be divided by $10^6$. The weights and the corrections are
symmetric in the labels. For each pair of labels the tables
give two continuous piecewise linear functions $u_0$ and $u_1$ through their
values at the knots $\frac78+\frac j8$, $0\le j\le11$; they are linear between
knots and vanish at $\frac34$, at $\frac{19}8$ and outside
$[\frac34,\frac{19}8]$. Then
$(B_0,\dots,B_5)=(u_0,\,u_1,\,-u_0-u_1,\,-u_0-u_1,\,u_1,\,u_0)$, which satisfies
\eqref{eq:zero} identically. For the second window, whose frames have eight
points, there are three functions and
$(B_0,\dots,B_6)=(u_0,\,u_1,\,u_2,\,-2(u_0+u_1+u_2),\,u_2,\,u_1,\,u_0)$.

\begin{table}
\caption{First window: pair weights, in units of $10^{-6}$. For each index distance the three
budgets are $2$, $4$ and $8$.}
\label{tab:weights}
\begin{tabular}{crrr}\toprule
$(i,j)$ & $W^{11}_{ij}$ & $W^{12}_{ij}$ & $W^{22}_{ij}$\\\midrule
$(0,1)$ & $290164$ & $613773$ & $937383$\\
$(0,2)$ & $632845$ & $943091$ & $1344283$\\
$(0,3)$ & $903375$ & $1903375$ & $3816137$\\
$(0,4)$ & $0$ & $853682$ & $1707365$\\
$(0,5)$ & $1000000$ & $2000000$ & $4000000$\\
$(0,6)$ & $2000000$ & $4000000$ & $8000000$\\
$(1,2)$ & $415442$ & $686130$ & $1865613$\\
$(1,3)$ & $238229$ & $639391$ & $1597288$\\
$(1,4)$ & $96624$ & $96624$ & $183862$\\
$(1,5)$ & $2000000$ & $2292634$ & $4585268$\\
$(1,6)$ & $1000000$ & $2000000$ & $4000000$\\
$(2,3)$ & $294392$ & $700096$ & $1197003$\\
$(2,4)$ & $257850$ & $835035$ & $2116857$\\
$(2,5)$ & $96624$ & $96624$ & $183862$\\
$(2,6)$ & $0$ & $853682$ & $1707365$\\
$(3,4)$ & $294392$ & $700096$ & $1197003$\\
$(3,5)$ & $238229$ & $639391$ & $1597288$\\
$(3,6)$ & $903375$ & $1903375$ & $3816137$\\
$(4,5)$ & $415442$ & $686130$ & $1865613$\\
$(4,6)$ & $632845$ & $943091$ & $1344283$\\
$(5,6)$ & $290164$ & $613773$ & $937383$\\
\bottomrule
\end{tabular}
\end{table}

\begin{table}
\caption{First window: site rewards, in units of $10^{-6}$, with sums $x_1=0.054350$ and
$x_2=0.069433$.}
\label{tab:rewards}
\begin{tabular}{crrrrrrr}\toprule
$i$ & $0$ & $1$ & $2$ & $3$ & $4$ & $5$ & $6$\\\midrule
$\kappa_{i,1}$ & $0$ & $27175$ & $0$ & $0$ & $0$ & $27175$ & $0$\\
$\kappa_{i,2}$ & $2063$ & $30080$ & $2054$ & $1039$ & $2054$ & $30080$ & $2063$\\
\bottomrule
\end{tabular}
\end{table}

\begin{table}
\caption{First window: values of $u_0$ and $u_1$ at the knots $\frac78+\frac j8$, in units of
$10^{-6}$. At the knots that are not listed all values are zero.}
\label{tab:knots}
\begin{tabular}{ccrrrrr}\toprule
Labels & & $j=1$ & $j=2$ & $j=7$ & $j=8$ & $j=9$\\\midrule
$11$ & $u_0$ & $186$ & $976$ & $-1914$ & $-1005$ & $0$\\
$11$ & $u_1$ & $-736$ & $-177$ & $124$ & $0$ & $0$\\
$12$ & $u_0$ & $334$ & $1483$ & $0$ & $-920$ & $-264$\\
$12$ & $u_1$ & $0$ & $-495$ & $0$ & $524$ & $111$\\
$22$ & $u_0$ & $2026$ & $2124$ & $0$ & $-432$ & $0$\\
$22$ & $u_1$ & $-924$ & $-11$ & $0$ & $0$ & $773$\\
\bottomrule
\end{tabular}
\end{table}

\begin{table}
\caption{Second window, frames of eight points: pair weights, in units of $10^{-6}$. For each
index distance their sum is at most $2$.}
\label{tab:line}
\begin{tabular}{crcrcrcr}\toprule
$(i,j)$ & $W_{ij}$ & $(i,j)$ & $W_{ij}$ & $(i,j)$ & $W_{ij}$ & $(i,j)$ & $W_{ij}$\\\midrule
$(0,1)$ & $159939$ & $(1,2)$ & $264496$ & $(2,4)$ & $614430$ & $(3,7)$ & $143766$\\
$(0,2)$ & $0$ & $(1,3)$ & $385569$ & $(2,5)$ & $757940$ & $(4,5)$ & $389188$\\
$(0,3)$ & $280438$ & $(1,4)$ & $340591$ & $(2,6)$ & $856233$ & $(4,6)$ & $385569$\\
$(0,4)$ & $143766$ & $(1,5)$ & $856233$ & $(2,7)$ & $709387$ & $(4,7)$ & $280438$\\
$(0,5)$ & $709387$ & $(1,6)$ & $581225$ & $(3,4)$ & $372750$ & $(5,6)$ & $264496$\\
$(0,6)$ & $1000000$ & $(1,7)$ & $1000000$ & $(3,5)$ & $614430$ & $(5,7)$ & $0$\\
$(0,7)$ & $2000000$ & $(2,3)$ & $389188$ & $(3,6)$ & $340591$ & $(6,7)$ & $159939$\\
\bottomrule
\end{tabular}
\end{table}

\begin{table}
\caption{Second window: values of $u_0$, $u_1$ and $u_2$ at the knots $\frac78+\frac j8$, in units
of $10^{-6}$. The only nonzero reward is $\kappa_3=8722$, and
$\alpha=627$.}
\label{tab:lineknots}
\begin{tabular}{crrrcrrr}\toprule
$j$ & $u_0$ & $u_1$ & $u_2$ & $j$ & $u_0$ & $u_1$ & $u_2$\\\midrule
$0$ & $647$ & $0$ & $-425$ & $6$ & $0$ & $0$ & $221$\\
$1$ & $555$ & $91$ & $-425$ & $7$ & $0$ & $0$ & $221$\\
$2$ & $518$ & $45$ & $-342$ & $8$ & $270$ & $86$ & $-275$\\
$3$ & $647$ & $0$ & $-425$ & $9$ & $325$ & $58$ & $-244$\\
$4$ & $647$ & $0$ & $-425$ & $10$ & $0$ & $0$ & $-425$\\
$5$ & $647$ & $0$ & $-425$ & $11$ & $647$ & $0$ & $-425$\\
\bottomrule
\end{tabular}
\end{table}

\begin{table}
\caption{Third window: pair weights, in units of $10^{-6}$. For each index distance their sum is
at most $2$.}
\label{tab:third}
\begin{tabular}{crcrcr}\toprule
$(i,j)$ & $W_{ij}$ & $(i,j)$ & $W_{ij}$ & $(i,j)$ & $W_{ij}$\\\midrule
$(0,1)$ & $263857$ & $(1,3)$ & $412303$ & $(2,6)$ & $608100$\\
$(0,2)$ & $364846$ & $(1,4)$ & $531836$ & $(3,4)$ & $394599$\\
$(0,3)$ & $468163$ & $(1,5)$ & $783798$ & $(3,5)$ & $412303$\\
$(0,4)$ & $608100$ & $(1,6)$ & $1000000$ & $(3,6)$ & $468163$\\
$(0,5)$ & $1000000$ & $(2,3)$ & $394599$ & $(4,5)$ & $341543$\\
$(0,6)$ & $2000000$ & $(2,4)$ & $445700$ & $(4,6)$ & $364846$\\
$(1,2)$ & $341543$ & $(2,5)$ & $531836$ & $(5,6)$ & $263857$\\
\bottomrule
\end{tabular}
\end{table}

\begin{table}
\caption{Third window: values of $u_0$ and $u_1$ at the knots $\frac78+\frac j8$, in units of
$10^{-6}$. The only nonzero reward is $\kappa_3=54525$, and
$\alpha=6212$.}
\label{tab:thirdknots}
\begin{tabular}{crrcrr}\toprule
$j$ & $u_0$ & $u_1$ & $j$ & $u_0$ & $u_1$\\\midrule
$0$ & $1642$ & $-224$ & $6$ & $565$ & $15503$\\
$1$ & $6659$ & $4991$ & $7$ & $4132$ & $5442$\\
$2$ & $7184$ & $4781$ & $8$ & $4638$ & $5326$\\
$3$ & $14319$ & $1750$ & $9$ & $5687$ & $5013$\\
$4$ & $19098$ & $-3030$ & $10$ & $0$ & $16033$\\
$5$ & $0$ & $-2543$ & $11$ & $0$ & $16068$\\
\bottomrule
\end{tabular}
\end{table}

\end{document}